\documentclass[12pt,reqno, psfig]{amsart}
\usepackage{mathrsfs}
\usepackage{amsgen}
\usepackage{amscd}
\usepackage{amsmath}
\usepackage{latexsym}
\usepackage{amsfonts}
\usepackage{amssymb}
\usepackage{amsthm}
\usepackage{graphicx}
\usepackage{color}
\usepackage{amsthm,amssymb}
\usepackage{graphicx}
\usepackage{enumerate}
\usepackage{amssymb,amsmath,graphicx,amsfonts,euscript}
\usepackage{color}
\usepackage{graphics}
\usepackage{mathrsfs}
\usepackage{float}
\usepackage[colorlinks=true,backref]{hyperref}

\usepackage[margin=3cm, a4paper]{geometry}

\def\H{{\cal H}}

\def\R{\mathbb{R}}

\def\H2{H^2(\R)}
\def\L2{L^2(\R)}

\newcommand{\dx}{\,\mathrm{d}x}

\newcommand{\sech}{\,\mathrm{sech}}

\def\H{{\cal H}}

\def\H1{H^1(\R)}

 \newcommand{\la}{\lambda}

 \newcommand{\Del}[1]{}

\numberwithin{equation}{section}

\newtheorem{thm}{Theorem}[section]
\newtheorem{cor}[thm]{Corollary}
\newtheorem{lem}[thm]{Lemma}
\newtheorem{claim}[thm]{Claim}
\newtheorem{prop}[thm]{Proposition}
\newtheorem{definition}[thm]{Definition}
\theoremstyle{remark}
\newtheorem{remark}[thm]{Remark}
\newtheorem*{exam*}{Examples}
\title[Instability of the solitary waves for the gBBM]{Instability of the solitary waves for the Generalized Benjamin-Bona-Mahony Equation}

\author{Rui Jia}
\address{(R. Jia) Center for Applied Mathematics\\
	Tianjin University\\
	Tianjin 300072, China}
\email{jiarui0305@tju.edu.cn}

\author{Yifei Wu}
\address{(Y.Wu) Center for Applied Mathematics\\
	Tianjin University\\
	Tianjin 300072, China}
\email{yerfmath@gmail.com}
\thanks{}

\date{}


\begin{document}
	\maketitle
		
	\setcounter{page}{1}

\begin{abstract}

In this work, we consider the generalized Benjamin-Bona-Mahony equation 
$$\partial_t u+\partial_x u+\partial_x( |u|^pu)-\partial_t \partial_x^{2}u=0, 
 \quad(t,x) \in \mathbb{R} \times \mathbb{R}, $$ 
with
 $p>4$.
This equation has the traveling wave solutions $\phi_{c}(x-ct), $ for any frequency $c>1.$ 
It has been proved by Souganidis and Strauss \cite{Strauss-1990} that,  
%for $0<p\leq 4, $ $\phi_{c}(x-ct)$ is orbitally stable for any $c>1$; 
%for $p>4, $ 
there exists a number $c_{0}(p)>1$, such that solitary waves $\phi_{c}(x-ct)$ with $1<c<c_{0}(p) $ is orbitally  unstable, while for $c>c_{0}(p), $ $\phi_{c}(x-ct)$ is orbitally stable. The linear exponential instability in the former case was further proved by Pego and Weinstein \cite{Pego-1991-eigenvalue}.
	 In this paper, we prove the orbital instability in the critical case $c=c_{0}(p)$. 
	  % $$c_{0}(p)=\frac{p}{4+2p}\left(1+\sqrt{2+\frac{1}{2}p}\right).$$
	   	\end{abstract}
   	
\section{Introduction}

It is well known that the KdV equation is a classical model used to describe the characteristics of water waves of long wave length in river channels. When studying nonlinear dispersive long wave unidirectional propagation, Benjamin, Bona, and Mahony \cite{T. B.benjamin-1972} considered a new model named the Benjamin-Bona-Mahony (BBM) equation, which can describe physical properties of long waves better. The BBM equation reads
\begin{align*}
	 u_t+u_x+uu_x-u_{txx}=0.
\end{align*}

 In this paper, we consider the following generalized Benjamin-Bona-Mohony (gBBM) equation
\begin{align}
	\label{1.1}
	\partial_t u+\partial_x u+\partial_x(|u|^pu)-\partial_t \partial_x^{2}u=0
	\quad(t,x) \in \mathbb{R} \times \mathbb{R}
\end{align}
with  $p>0. $ For $H^1$-solution, the momentum $Q$ and the energy $E$ are conserved under the flow, where
  \begin{align}
  	\label{momentum Q}
  	&Q(u)=\frac{1}{2}\int_{\mathbb{R}} u^{2}+u_{x}^{2} \dx;\\
  	\label{energy E}
  	&E(u)=\frac{1}{2}\int_{\mathbb{R}} u^{2}\dx+\frac{1}{p+2}\int_{\mathbb{R}}|u|^{p+2} \dx.
  \end{align}
 In particular, the equation \eqref{1.1} can be expressed in the following Hamiltonian form
\begin{align}
	\label{equa-hami}
	\partial_t u= JE'(u),\qquad \mbox{ where } \qquad J=-(1-\partial_x^2)^{-1}\partial_x.
\end{align}
 In \cite{Strauss-1990}, Souganidis and Strauss proved that if $u_{0}=u(0,x)\in \H1, $ there exists a unique global solution $u$ of \eqref{1.1} in $C(\R; \H1)$. 
  
The equation \eqref{1.1} has the solitary waves solution given by
$u(x, t)=\phi_{c}(x-ct)$ for any $c>1$, 
where $\phi_{c}$ is the ground state solution of the following elliptic equation
\begin{align}
	\label{1.3}
	-c\partial_{xx}\phi_{c}+(c-1)\phi_{c}-\phi_{c}^{p+1}=0.
\end{align}
The ground state solution $\phi_{c}$ is a smooth, even, and positive function, which decays exponentially as $|x|\rightarrow \infty $, in the sense that $|\phi_{c}|\leq C_{1}e^{-C_{2}|x|},  |\partial_x\phi_{c}|\leq C_{1}e^{-C_{2}|x|}$ 
for some $C_{1}, C_{2}>0. $ Then a natural problem is the stability theory of the solitary waves solution $\phi_{c}(x-ct)$, which is defined as follows. 
For $\varepsilon>0, $ we denote the set $	U_{\varepsilon}(\phi_{c})$ as:
  \begin{align}
  	\label{1.9}
  	U_{\varepsilon}(\phi_{c})&=
\{
u \in  C(\R; \H1): \sup_{t \in \R}\inf_{y\in\mathbb{R}}\big\|u-\phi_{c}(\cdot-y)\big\|_{\H1}<\varepsilon
\}.
  \end{align}
Then we define the orbital stability/instability of the solitary waves as 
\begin{definition}
	We say that the solitary waves solution $\phi_{c}(x-ct)$ of \eqref{1.1} is orbitally stable if for any $\varepsilon>0, $ there exists $\delta>0$ such that if 
	$\big\|u_{0}-\phi_{c}\big\|_{\H1}<\delta, $
	then the solution $u$ of \eqref{1.1} with
	 $u(0, x)=u_{0}(x)$
	 satisfies
	 $u \in U_{\varepsilon}(\phi_{c}). $ 
	 Otherwise, $\phi_{c}(x-ct)$ is said to be orbitally unstable.
\end{definition}

%Theorem 1.6.
Regarding the stability theory of these solitary waves, Souganidis and Strauss \cite{Strauss-1990} proved that when $0<p\leq 4$, the solitary waves solution $\phi_{c}(x-ct)$ is orbitally stable for all $c>1$, while when $p>4$, the solitary waves solution $\phi_{c}(x-ct)$ is orbitally unstable in $\H1$ for $1<c<c_{0}(p)$ and orbitally stable in $\H1$ for $c>c_{0}(p)$. Here
$$
c_{0}(p):=\frac{p}{4+2p}\left(1+\sqrt{2+\frac{1}{2}p}\right).
$$
Denote $d(c):=E(\phi_{c})-cQ(\phi_{c})$, the critical parameter $c_0(p)$ is determined by 
$$
d''(c)\Big|_{c=c_0(p)}=0.
$$
Since the operator $J$ is not onto, the framework of Grillakis, Shatah, and Strauss \cite{GrShStr-87,GrShStr-90} cannot be directly applied to study the stability of the solitary waves $\phi_{c}(x-ct)$. Therefore, the work \cite{Strauss-1990} is not a direct application of the theories established in \cite{GrShStr-87,GrShStr-90}. For further discussion on these cases, readers are referred to the more recent paper \cite{LZ-2022-MAMS} by Lin and Zeng.

In \cite{Pego-1991-eigenvalue}, Pego and Weinstein established criteria for the linear exponential instability of solitary waves solution of the gBBM equation. They further proved the linear exponential instability of $\phi_{c}(x-ct)$ for each $p>4$ when $1<c<c_{0}(p)$.

So far, the stability of the solitary waves $\phi_{c}(x-ct)$ has been nearly established, except for the critical case $c=c_0(p)$, which corresponds to the degenerate case where $d''(c)=0$. In this paper, our aim is to fully establish the stability of the solitary waves $\phi_{c}(x-ct)$ by studying the degenerate case of $c=c_0(p)$.

Before presenting our theorem, let us clarify some definitions that will be used. 
We define the functional $S_c$ as
\begin{align}
	\label{S_c}
	S_c(u):=E(u)-cQ(u).
\end{align}
Then the equation \eqref{1.3} is equivalent to $S_c'(\phi_{c})=0$.
For convenience, we denote that 
$$\omega=c^{-\frac12},$$
and
\begin{align}
	\label{1.4}
	\psi_{\omega}(x)=c^{-\frac{1}{p}}\phi_{c}(x).
\end{align}
Then by \eqref{1.3}, we find that $\psi_{\omega}$ satisfies the following equation:
\begin{align}
	\label{1.5}
	-\partial_{xx}\psi_{\omega}+(1-\omega^{2})\psi_{\omega}-\psi_{\omega}^{p+1}=0.
\end{align}

As mentioned earlier, our objective is to demonstrate the instability of the solitary waves in the critical frequency case: $c=c_0(p)$. 
Our argument relies on the assumption of the negativity of a specific direction of the Hessian operator $S_c''(\phi_{c})$, which is confirmed numerically. 
More precisely, let  
 	\begin{align}
 		\label{Psi_c}
	        \Psi_{c}&:=-\frac{1}{2}\omega^{1-\frac{2}{p}}\partial_{\omega} \psi_{\omega},\\
 		\label{Gamma_c}
 		\Gamma_c
 		&:=B(c)\left[c^2\Psi_{c}+\frac{c}{2}x\partial_{x}\phi_{c}+c\phi_{c}\right]+D(c)(3x^2\phi_{c}+x^3\partial_{x}\phi_c),
 	\end{align}
 	where 
 	\begin{align*}
 		&D(c)=-\frac{4pc+4c-3p}{2(p+4)}\big\|\phi_{c}\big\|^2_{L^2},
 		\quad
 		B(c)=\frac{3}{2}\big\|x\phi_{c}\big\|^2_{L^2}+\frac{9}{2}\big\|x\partial_{x}\phi_{c}\big\|^2_{L^2}-3\big\|\phi_{c}\big\|^2_{L^2}.
 	\end{align*}
 	We assume that 
 	\begin{align}
 		\label{second-orth-condition}
 		\langle S''_c(\phi_{c})\Gamma_c, \Gamma_c \rangle <0 \quad  \mbox{ holds for }  c=c_0(p) \mbox{ with }p>4,
 	\end{align}
which is checked numerically
 	\footnote{According to Appendix \ref{numerical result}, we use Matlab to compute  $\langle S''_c(\phi_{c})\Gamma_{c}, \Gamma_{c}\rangle <0. $ It suffices in practice to run the computations until $p=100$ to check $\langle S''_c(\phi_{c})\Gamma_{c}, \Gamma_{c}\rangle <0 $ as the inner product is decreasing fastly as a power function when $p$ is bigger than $10. $ We refer to ``Appendix \ref{numerical result}" for more details. } 
 	in Appendix \ref{numerical result}.

  The main result in the present paper is the following.
 \begin{thm}\label{main:thm}
 	Let $p>4, c=c_0(p)$ and $\phi_{c}$ be the ground state of \eqref{1.3}. 
 	Assume \eqref{second-orth-condition}, then the solitary waves solution $\phi_{c}(x-ct)$ is orbitally unstable.
 \end{thm}
 
 \begin{remark}
 (1)  In the previous works \cite{Pego-1991-eigenvalue} and \cite{Strauss-1990}, the stability and instability results of solitary waves solution $\phi_{c}(x-ct)$ for the non-degenerate case $d''(c)\neq 0$ in the gBBM equation have already been established. Under \eqref{second-orth-condition}, Theorem \ref{main:thm} closes the only remaining gap for the degenerate case $d''(c)= 0$ and thereby completes the entire stability theory of solitary waves solution for the gBBM equation. We also give  an element numerical computation to check \eqref{second-orth-condition}.

(2) The instability of the solitary waves $\phi_{c}(x-ct)$ with $1<c<c_{0}(p)$ has been demonstrated using the Lyapunov stability argument based on the monotonicity of the Lyapunov functional in the non-degenerate cases $d''(c)\neq 0. $ However, this argument does not apply to the degenerate cases, as the Lyapunov functional loses monotonicity when $d''(c)=0. $ Therefore, we need to construct a new monotonic functional. The outline of the proof will be provided in the following subsection.
  
\end{remark}

\subsection{Sketch of the proof}
The main approach is to construct a monotonic quantity based on virial quantities and the modulation argument, drawing inspiration from \cite{wu-KG}, which established the instability of standing wave solutions of the Klein-Gordon equation in the degenerate case. The methodology involves analyzing the orthogonality conditions and the dynamics of the modulated parameters. However, due to the intricate structure of the gBBM equation, constructing the monotonic functional in this paper is much more complex than the Klein-Gordon equation case. In particular, the non-onto property of the skew symmetry operator $J$ poses significant obstacles. The key ingredients of the proof can be summarized as follows.

  {\it Step1: Modulation.} First of all, we assume the solitary wave is stable. The modulation argument allows us to find two parameters, $y(t)$ and $\lambda(t), $ and a perturbation function $\xi$, such that the solution $u$ can be expressed as 
  \begin{align}
  	\label{perturbation-fun}
  	u(t,x)=(\phi_{\lambda}+\xi)\big(x-y(t)\big),
  \end{align}
 where $\lambda(t)$ is a scaling parameter suitably defined, and $y(t)$ is a spatial translation parameter.
 We also need to find two different orthogonality conditions, namely: 
  \begin{align}
  	\label{intro-orth-vague}
  	\langle \xi, \psi_1 \rangle=\langle \xi, \psi_2 \rangle=0.
  \end{align} 
To find suitable $\psi_1$ and $\psi_2, $ it is natural to consider the spectrum of the Hessian of the action $S_{c}''(\phi_{c})$.
 The study of \cite{weinstein-1985-modulational} indicates that $\ker S''_c(\phi_{c})=\{\alpha\partial_{x}\phi_{c}, \alpha\in \R\}$ and $S''_c(\phi_{c})$ has a unique negative eigenvalue.
 The properties of spectrum of $S''_c(\phi_{c})$ helpfully identify the origin of $\psi_1$ and $\psi_2, $ specifically:
 \begin{align}
 	\label{orth-orign}
 	\psi_1 \in \ker S''_c(\phi_{c}), \quad \langle S''_c(\phi_{c})\psi_2, \psi_2\rangle <0.
 \end{align}
 It is worth noting that $\psi_2$ is not unique and it is not necessary to choose the negative eigenfunction. 
Indeed, the choice of $\psi_2$ is crucial as its concrete expression has a significant impact on the construction of monotonicity, which will be addressed in Step 6 below. 
%This constitutes the first key ingredient in our proof. 

 {\it Step2: Coercivity.}
 Having determined the properties of $\psi_1$ and $\psi_2 $ (i.e.,\eqref{orth-orign}), we shall prove the the coercivity of the Hessian $S_{\lambda}''(\phi_{\lambda})$ as shown in Proposition \ref{prop3.4} for a general criterion by means of spectral decomposition argument, which can be expressed as follows:
 \begin{align*}
  \langle S_{\lambda}''(\phi_{\lambda})\xi, \xi\rangle\gtrsim \|\xi\|^2_{H^1}.
 \end{align*}
 In addition, in the degenerate case $c= c_0(p)$, we have 
 \begin{align}\label{flatness}
  \langle S_{c}''(\phi_{c})\partial_{c}\phi_{c}, \partial_{c}\phi_{c}\rangle= 0.
  \end{align}
This flatness equality \eqref{flatness}, combined with the coercivity of $S_{\lambda}''(\phi_{\lambda})$, implies an important estimate: 
 \begin{align}\label{control-xi}
 \|\xi\|_{H^1}\ll |\lambda-c|.
 \end{align}
 This means that the perturbation of the solution $\xi, $ can be controlled by the scaling increment $\lambda-c$. 
%This constitutes the second key ingredient in our proof.

 We emphasize that the Step 1 and Step 2 here are similar to the paper in \cite{wu-KG}, however, the following steps are of much problem-dependence and much more complicated for BBM mainly due to its poor Hamiltonian structure.

 {\it  Step 3: Dynamic of the modulation parameters.}
  Directly following the Implicit Function Theorem, the (translation) modulation parameter $y$ has a trivial bound given by 
 \begin{align}\label{rough-y-dot}
 	\dot y-\lambda=O(\|\xi\|_{H^1}).
 \end{align}
In simpler terms, $\dot y-\lambda$ is actually the first-order of $\xi.$
However, the rough estimate is not enough to support the later analysis.
To obtain a more accurate estimate, we apply \eqref{perturbation-fun} to \eqref{equa-hami}, which yields 
 \begin{align}
 	\label{intro-dynamic}
 	 \dot \xi+\partial_\lambda \phi_\lambda \cdot \dot \lambda-\partial_x \phi_\lambda\cdot (\dot y-\lambda)
 	=JS_\lambda''(\phi_\lambda)\xi+\mbox{``high-order  term''}. 
 \end{align}
 This gives us the key expression of the dynamic of $\dot y-\lambda$: 
 \begin{align}\label{dot-y-Roughdynamic}
 	 \dot y-\lambda= c_1(\lambda)\left\langle \xi, S_\lambda''(\phi_\lambda)\big(Jf_{\lambda}\big)\right\rangle+c_1(\lambda)\partial_t \big\langle \xi, f_{\lambda}\big\rangle+\mbox{``high-order  term''},
 \end{align}
 for any $f_{\lambda}$ satisfying 
 \begin{align}
 	\label{condition-f_lambda}
  \langle f_{\lambda},\partial_\lambda \phi_\lambda\rangle=0,\quad \langle f_{\lambda},\partial_x \phi_\lambda\rangle\ne 0.
 \end{align}
Obviously, the function $f_{\lambda}$ satisfying condition \eqref{condition-f_lambda} is not unique. The estimate \eqref{dot-y-Roughdynamic} is therefore relatively flexible, depending on the choice of $f_\lambda$. Indeed, the latter almost determines the expression of the first-order of $\xi$ which appears in $I'(t)$ defined later. 
This constitutes the first key ingredient in our proof.

{\it Step 4: Design of the virial identity.}  
We are now in the position to consider the construction of virial identity $I(t). $ Our goal is to show that it exhibits monotonic behavior, i.e., $I'(t)>0$ or $I'(t)<0. $ The virial identity typically arises from conservation laws such as \eqref{momentum Q}-\eqref{energy E} and the dynamic of the modulated function as in \eqref{intro-dynamic}. The ideal form of $I'(t)$ is as follows:
\begin{align}
	\label{idealform}
	I'(t)= \beta(u_{0})+\gamma(\lambda)+\mbox{``high-order term"},
\end{align}
where the high-order term is in fact $\|\xi\|^{2}_{H^1} $ that has been estimated in Step 2. 
If $\beta(u_0)$ is a positive quantity,
and $\gamma(\lambda)$ is also a positive quantity satisfying $\gamma(\lambda)\gtrsim(\lambda-c)^2 $ which requires that \begin{align}\label{positive-gamma(lambda)}
	\gamma(c)=\gamma'(c)=0 \quad \mbox{and}\quad \gamma''(c)>0,
\end{align}
 then the monotonicity of virial identity is guaranteed. This constitutes the second key ingredient in our proof.

 {\it Step 5: Construction of the monotonicity.} 
 Unlike the Lyapunov functional, the main monotonic functional here comes from the localized virial identity. Specifically, we first define 
$$
I(t)=\int_{\R} \chi(x-y(t))(\frac{1}{2}u^{2}+\frac{1}{p+2}|u|^{p+2}) \dx.
$$
where $\chi$ is a suitable smooth cutoff function. 
By the expansion \eqref{perturbation-fun}, we observe that $I'(t)$ has the following structure: 
$$
I'(t)=\beta(u_{0})+\gamma(\lambda)+c_2(\lambda)\cdot(\dot{y}-\lambda)+c_3(\lambda)\langle \xi, \phi_{\lambda}+\partial_{xx}\phi_{\lambda}\rangle+c_4(\lambda)\langle \xi, S''_{\la}(\phi_{\lambda})\phi_{\lambda}\rangle
+O\big(\|\xi\|^2_{H^1}\big).
$$
 Since the precise estimate of $\dot{y}-\lambda $ is already known in Step 3, we obtain the structure of $I'(t)$ as follows:
\begin{align*}
	I'(t)
	&=
	 \beta(u_{0})+\gamma(\lambda)
	+\langle \xi, S''_{\lambda}(\phi_{\lambda})\big[c_1(\lambda)c_2(\lambda)Jf_{\lambda}+c_4(\lambda)\phi_{\lambda}\big] \rangle 
	 \\
	&\quad
	+c_1(\lambda)c_2(\lambda)\partial_t \big\langle \xi, f_{\lambda}\big\rangle
	+c_3(\lambda)\langle \xi, \phi_{\lambda}+\partial_{xx}\phi_{\lambda}\rangle
	+O\big(\|\xi\|^2_{H^1}\big).
\end{align*}
Move the term $c_1(\lambda)c_2(\lambda)\partial_t \big\langle \xi, f_{\lambda}\big\rangle$  to the left-hand side, and we further obtain that 
\begin{align}
	\label{stru-I'(t)}
	\frac{d}{dt}\Big(I(t)-c_1(\lambda)c_2(\lambda)\big\langle \xi, f_{\lambda}\big\rangle\Big)
	&=
	 \beta(u_{0})+\gamma(\lambda)
	+\langle \xi, S''_{\lambda}(\phi_{\lambda})\big[c_1(\lambda)c_2(\lambda)Jf_{\lambda}+c_4(\lambda)\phi_{\lambda}\big] \rangle 
	\notag\\
	&\quad
	+c_3(\lambda)\langle \xi, \phi_{\lambda}+\partial_{xx}\phi_{\lambda}\rangle
	+O\big(\|\xi\|^2_{H^1}\big).
\end{align}
This inspires us to make a bold assumption: After suitably choosing $f_{\lambda}$, if 
\begin{align}
	\label{intro-orth}
	\left\langle \xi, S''_{\lambda}(\phi_{\lambda})\big[c_1(\lambda)c_2(\lambda)Jf_{\lambda}+c_4(\lambda)\phi_{\lambda}\big]
	+c_3(\lambda)\big(\phi_{\lambda}+\partial_{xx}\phi_{\lambda}\big)\right\rangle=0, 
\end{align} 
% exactly,
%  or \begin{align*}
%  S''_{\lambda}(\phi_{\lambda})\big[c_1(\lambda)c_2(\lambda)Jf_{\lambda}+c_4(\lambda)\phi_{\lambda}\big]
%  +c_3(\lambda)\big(\phi_{\lambda}+\partial_{xx}\phi_{\lambda}\big)=\psi_2
%  \end{align*}
%in \eqref{intro-orth-vague}, 
  then $I'(t)$ will become the ideal form \eqref{idealform}. 
In order to match the form in \eqref{orth-orign}, we need the existence and the explicit expressions of the pre-images of $\phi_{\lambda} $ and $\partial_{xx}\phi_{\lambda}$. 
As a matter of fact, we need to find $\Psi_1$ and $\Psi_2, $ such that
\begin{align*}
	\phi_{\lambda}=S_{\lambda}''(\phi_{\lambda})\Psi_1, \quad
	 \partial_{xx}\phi_{\lambda}=S_{\lambda}''(\phi_{\lambda})\Psi_2.
\end{align*}
It is not an easy task, but we accomplished it. In fact, we observe that
\begin{align*}
	\phi_{\lambda}&=S_{\lambda}''(\phi_{\lambda})\big(-\frac{1}{2}\omega^{1-\frac{2}{p}}\partial_{\omega}\psi_{\omega}\big), 
	\\
	\partial_{xx}\phi_{\lambda}&=S_{\lambda}''(\phi_{\lambda})\big(\frac{1}{2}x\partial_x\phi_{\lambda}\big),
\end{align*}
where $\omega=\lambda^{-\frac{1}{2}}, $ $\psi_{\omega}$ satisfying \eqref{1.5}.
This constitutes the third key ingredient in our proof.

{\it Step 6: Verification of the negative direction.} 
Inspired by \eqref{intro-orth}, 
we denote 
$$
\Upsilon_{\lambda}
=c_1(\lambda)c_2(\lambda)Jf_{\lambda}+c_3(\lambda)\big(-\frac{1}{2}\omega^{1-\frac{2}{p}}\partial_{\omega}\psi_{\omega}+\frac{1}{2}x\partial_x\phi_{\lambda}\big)+c_4(\lambda)\phi_{\lambda}. 
$$
It is time to verify the negativity of $S''_{\lambda}(\phi_{\lambda})$ on $\Upsilon_{\lambda}$. More precisely, the problem finally  reduces to the following claim: 
\begin{claim}
There exists a function $f_{\lambda}$ verifying \eqref{condition-f_lambda}, such that 
\begin{align}
	\label{negative direction}
	\langle S''_{\lambda}(\phi_{\lambda})\Upsilon_{c},\Upsilon_{c}\rangle<0.
\end{align}
\end{claim}
This claim is established by choosing $f_\lambda=(1-\partial_x^2)(x^3\phi_{\lambda})$  as presented in the assumption \eqref{second-orth-condition}. Due to the complexity of the expression $\Upsilon_{\lambda}$, we decide to check it by numerical experiments. Suppose the claim is true, then we choose $\psi_2=S''_{\lambda}(\phi_{\lambda})\Upsilon_{\lambda}$ and thus obtain the form of $I'(t)$ in \eqref{idealform} as we expected. 
 This constitutes the fourth key ingredient in our proof.

{\it Step 7: Contradiction.} 
Based on the works before, the structure of monotonicity becomes clear. Indeed, 
$$I'(t)= \beta(u_{0})+\gamma(\lambda)+O(\|\xi\|^2_{H^1}). $$
Then using \eqref{control-xi}, we can infer that 
$$
  I'(t)\ge \beta(u_{0})+
\frac12 C(\lambda-c)^{2}.
$$
The positivity of $\beta(u_{0})$ can be verified by suitably choosing initial data $u_0. $
Thus we establish the monotonicity of $I(t).$ 
The contradiction between uniformly boundedness and monotonicity of $I(t)$ proves the instability in the end.

 \subsection{Organization of the paper}
The remainder of the paper is organized as follows. In Section \ref{2}, we provide some preliminaries. In Section \ref{3}, we establish the coercivity of the Hessian $S_\lambda''(\phi_{\lambda})$ and control the modulation parameters. In Section \ref{4}, we demonstrate the localized virial identities and define the monotonicity functional. In Section \ref{5}, we establish the monotonicity of the functional obtained in Section \ref{4} and prove the main theorem. Finally, in Appendix \ref{appendix}, we present a general coercivity property of the Hessian of the action $S_c''(\phi_{c})$ and the numerical result of the negative eigenfunction of $S_c''(\phi_{c})$.
 
  \section{Notations}
  \label{2}
  \subsection{Notations}
    For 
  $f, g\in \L2, $
  we define
  $$
  \langle f,g\rangle =\int_{\R}f(x)g(x)\dx
  $$
  and regard $\L2$ as a real Hilbert Space. 
  For a function $f(x)$, its $L^{q}$-norm
  $\|f\|_{q} =\left(\int_{\R}|f(x)|^{q} \dx\right)^{\frac{1}{q}}$
  and its $H^{1}$-norm 
  $\|f\|_{H^{1}}=(\|f\|^{2}_{L^{2}}+\|\partial_x f\|^{2}_{L^{2}})^\frac{1}{2}.$

    Further, we write $X\lesssim Y$ or $Y\lesssim X$ to indicate $X\leqslant CY$ for some constant $C>0. $
    We use the notation $X\sim Y$ to denote $X\lesssim Y\lesssim X. $
    We also use $O(Y)$ to denote any quantity $X$ such that $|X|\lesssim Y $ and use $o(Y)$ to denote any quantity $X$ such that $X/Y\rightarrow 0$ if $Y\rightarrow 0. $
    Throughout the whole paper, the letter $C$ will denote various positive constants which are of no importance in our analysis.
\subsection{Some basic definitions and properties}
In the rest of this paper, we consider the case of $p>4,$ and 
 $
 c=c_{0}(p).
 %=\frac{p}{4+2p}\left(1+\sqrt{2+\frac{1}{2}p}\right).
 $
 Recall the expression of conserved equality and the functional $S_{c} $, we have
  \begin{align} 
  	E(u)
  &=\frac{1}{2}\int_{\mathbb{R}} u^{2}\dx+\frac{1}{p+2}\int_{\mathbb{R}}|u|^{p+2} \dx,
  \notag\\
  	Q(u)
  	&=\frac{1}{2}\int_{\mathbb{R}}\big( u^{2}+u_{x}^{2}\big) \dx;
  \notag	\\
  	\label{2.5}
  	S_{c}(u)&=E(u)-cQ(u).
  \end{align}
Taking derivative, then we have
\begin{align}
\label{2.6}
	E'(u)
&=u+|u|^{p}u,\\
\label{2.7}
	Q'(u)
&=u-\partial_{xx}u,\\
S_{c}'(u)
&=E'(u)-cQ'(u)=c\partial_{xx} u +(1-c)u+|u|^{p}u.\notag
\end{align}
Note that 
$
S_{c}'(\phi_{c})=0. 
$
Moreover, for the real-valued function $f, $ a direct computation shows
\begin{align}
	\label{2.8}
	S_{c}''(\phi_{c})f
	=c\partial_{xx} f+(1-c)f+(p+1)\phi_{c}^{p}f.
\end{align}
Taking the derivative of 
$S_{c}'\big(\phi_{c}(\cdot-x)\big)=0 $ with respect to $x$ gives
\begin{align}
	\label{2.9}
	S_{c}''(\phi_{c})(\partial_x \phi_{c})=0.
\end{align}
For any function $\xi, \eta, $ we have
\begin{align}
	\label{2.10}
	\langle S_{c}''(\phi_{c})\xi, \eta\rangle 
=\langle S_{c}''(\phi_{c})\eta, \xi\rangle.
\end{align}
Moreover, taking the derivative of 
$S_{c}'(\phi_{c})=0 $ with respect to $c$ gives
\begin{align}
	\label{2.11}
	S_{c}''(\phi_{c})\partial_{c}\phi_{c}
=Q'(\phi_{c}).
\end{align}

  Next, we give some basic properties on the momentum, energy and the functional $S_c$.
  \begin{lem}
  	\label{lem2.1}
  	Let $ c=c_{0}(p); $ then the following equality holds:
  \begin{align*}
  	\partial_{c}Q\left(\phi_{c}\right)\Big|_{c=c_{0}(p)}=0.
  \end{align*}
  \end{lem}
\begin{proof}
	Note that
\begin{align*}
	Q(u)
	=\frac{1}{2}\int_{\mathbb{R}} u^{2}+u_{x}^{2} \dx. 
\end{align*}
Taking inner product of \eqref{1.3} and
$\phi_{c}, x\partial_{x}\phi_{c}$ respectively, by integration-by-parts, we can get
\begin{align*}
&c\big\|\partial_{x}\phi_{c}\big\|^{2}_{L^{2}}+(c-1)\big\|\phi_{c}\big\|^{2}_{L^{2}}-\big\|\phi_{c}\big\|^{p+2}_{L^{p+2}}=0,\\
&c\big\|\partial_{x}\phi_{c}\big\|^{2}_{L^{2}}-(c-1)\big\|\phi_{c}\big\|^{2}_{L^{2}}+\frac{2}{p+2}\big\|\phi_{c}\big\|^{p+2}_{L^{p+2}}=0.
\end{align*}
This gives that 
\begin{align}
	\label{2.1}
\big\|\phi_{c}\big\|^{p+2}_{L^{p+2}}
=\frac{2(p+2)(c-1)}{p+4}\big\|\phi_{c}\big\|^{2}_{L^{2}};
\quad
%\label{2.2}
\big\|\partial_{x}\phi_{c}\big\|^{2}_{L^{2}}
=\frac{p(c-1)}{(p+4)c}\big\|\phi_{c}\big\|^{2}_{L^{2}}.
\end{align}
 This further yields that 
$$
Q(\phi_{c})
=\frac{1}{2} \left[1+\frac{p(c-1)}{(p+4)c} \right]\big\| \phi_{c}\big\| ^{2}_{L^{2}}.
$$
By scaling, we find
\begin{align}
	\label{2.3}
	\big\| \phi_{c}\big\| ^{2}_{L^{2}}
&=c^{\frac{1}{2}}(c-1)^{\frac{2}{p}-\frac{1}{2}}\big\|\psi_{0}\big\|^{2}_{L^{2}},
\end{align}
where $\psi_0$ is the solution of 
\begin{align*}
	-\partial_{xx}\psi_0+\psi_0-\psi_0^{p+1}=0.
\end{align*}
Hence,
\begin{align}
	\label{2.4}
	\partial_{c}\big\| \phi_{c}\big\| ^{2}_{L^{2}}
=\frac{4c-p}{2pc(c-1)}\big\| \phi_{c}\big\| ^{2}_{L^{2}}.
\end{align}
 By a straightforward computation, we have
 \begin{align}
 	\label{partial_c_Q(phi_c)}
 	 \partial_{c} Q(\phi_{c})
 	=\frac{8(p+2)c^{2}-8pc-p^{2}}{4p(p+4)c^2(c-1)}
 	\big\|\phi_{c}\big\|^{2}_{L^{2}}. 
 \end{align}
 Finally, we substitute 
 $
c=c_{0}(p)
 $
 into the equality above, and thus we complete the proof.
\end{proof}

Then a consequence of Lemma \ref{lem2.1} is 
\begin{cor}
	\label{cor2.2}
	Let $\lambda>1, c=c_{0}(p)$, then
\begin{align*}
	S_{\lambda}(\phi_{\lambda})-S_{\lambda}(\phi_{c})
	=o((\lambda-c)^{2}).
\end{align*}
\end{cor}
\begin{proof}
	From the definition of $S_{c}$ in \eqref{2.5}, we have
\begin{align}
	\label{2.14}
	S_{\lambda}(\phi_{\lambda})-S_{\lambda}(\phi_{c})
	=S_{c}(\phi_{\lambda})-S_{c}(\phi_{c})
	  -(\lambda-c)\big[Q(\phi_{\lambda})-Q(\phi_{c})\big]. 
\end{align}
Recall that 
$
S_{c}'(\phi_{c})=0, 
$
then we use Taylor's expansion to calculate
\begin{align}
	\label{2.15}
	S_{\lambda}(\phi_{\lambda})-S_{\lambda}(\phi_{c})
	=\frac{1}{2}\langle S_{c}''(\phi_{c})(\phi_{\lambda}-\phi_{c}),  (\phi_{\lambda}-\phi_{c})\rangle 
	-(\lambda-c)\big[Q(\phi_{\lambda})-Q(\phi_{c})\big]
	+o((\lambda-c)^{2}). 
\end{align}
Note that
\begin{align*}
	\phi_{\lambda}-\phi_{c}
	=(\lambda-c)\partial_{c} \phi_{c}+o(\lambda-c), 
\end{align*}
then we find
\begin{align*}
	\langle S_{c}''(\phi_{c})(\phi_{\lambda}-\phi_{c}),  (\phi_{\lambda}-\phi_{c})\rangle 
    &=(\lambda-c)^{2}\langle S_{c}''(\phi_{c})\partial_{c} \phi_{c}, \partial_{c}\phi_{c}\rangle +o((\lambda-c)^{2})\\
    &=(\lambda-c)^{2}
    \langle Q'(\phi_{c}), \partial_{c}\phi_{c}\rangle 
    +o((\lambda-c)^{2})\\
    &=(\lambda-c)^{2}\cdot\partial_{c}
    Q(\phi_{c})\big|_{c=c_{0}(p)}
     +o\big((\lambda-c)^{2}\big),
\end{align*}
where we used \eqref{2.11} in the second step. Using Lemma \ref{lem2.1}, we have
\begin{align*}
\partial_{c}Q(\phi_{c})|_{c=c_{0}(p)}=0. 
\end{align*}
Hence, 
\begin{align*}
	Q(\phi_{\lambda})-Q(\phi_{c})=o(\lambda-c), 
\end{align*}
and
\begin{align*}
	\langle S_{c}''(\phi_{c})(\phi_{\lambda}-\phi_{c}),  (\phi_{\lambda}-\phi_{c})\rangle 
	=o\left((\lambda-c)^{2}\right). 
\end{align*} 
Taking these two results into \eqref{2.15}, we obtain the desired estimate. 
\end{proof}

%\begin{align}
%	\label{1.6}
%	\psi_{\omega}=(1-\omega^{2})^{\frac{1}{p}}\psi_{0}(\sqrt{1-\omega^{2}}x)
%\end{align}
%By multiplying \eqref{1.5} with $\psi_{\omega}, x\partial_{x}\psi_{\omega}$ respectively, integrating by parts, we obtain
%\begin{align}
%	\label{psi-psi_x}
%	\|\psi_{\omega}\|^{p+2}_{L^{p+2}}=\frac{2(p+2)(1-\omega^2)}{p+4}\|\psi_{\omega}\|^2_{L^2};
%\quad
%   \|\partial_{x}\psi_{\omega}\|^{2}_{L^{2}}=\frac{p(1-\omega^2)}{p+4}\|\psi_{\omega}\|^2_{L^2}.
%\end{align}

 The next lemma gives pairs of pre-image and image of $S_{c}''(\phi_{c}). $
 \begin{lem}
 	\label{lem3.3}
It holds that 
 \begin{align}
\label{3.5}
S_{c}''(\phi_{c})(x\partial_{x}\phi_{c})
&=2c\partial_{xx}\phi_{c}.
 \end{align}
Moreover,  let $\psi_{\omega}$ be defined in \eqref{1.4}, and denote that
 $$\Psi_{c}=-\frac{1}{2}\omega^{1-\frac{2}{p}}\partial_{\omega} \psi_{\omega}.
 $$
 Then 
 \begin{align}
\label{3.6}
S_{c}''(\phi_{c})\Psi_{c}
&=\phi_{c}.
 \end{align}
 \end{lem}
\begin{proof}
First, by the expression of $S_{c}''(\phi_{c})$ in \eqref{2.8},  we have
\begin{align*}
	S_{c}''(\phi_{c})(x\partial_{x}\phi_{c})
	&=c\partial_{xx}(x\partial_{x}\phi_{c})+(1-c)(x\partial_{x}\phi_{c})+(p+1)\phi_{c}^{p}(x\partial_{x}\phi_{c})\\
	&=2c\partial_{xx}\phi_{c}+x\partial_{x}(c\partial_{xx}\phi_{c}+(1-c)\phi_{c}+\phi_{c}^{p+1})\\
	&=2c\partial_{xx}\phi_{c}.
\end{align*}
So we obtain \eqref{3.5}.

Second, taking the derivative of \eqref{1.5} with respect to $\omega, $ we have
\begin{align}
	\label{3.7}
	-\partial_{xx}(\partial_{\omega} \psi_{\omega})
	+(1-\omega^{2})\partial_{\omega}\psi_{\omega}
	-(p+1)\psi_{\omega}^{p}\partial_{\omega}\psi_{\omega}
	=2\omega \psi_{\omega}.
\end{align}
Moreover, denoting that
\begin{align}
	\label{2.12}
	L_{\omega}f
	=-\partial_{xx} f+(1-\omega^{2})f-(p+1)\psi_{\omega}^{p}f, 
\end{align}
then we have that 
\begin{align}
	\label{2.13}
S_{c}''(\phi_{c})f=-cL_{\omega}f.
\end{align}
Using the expression of $S_{c}''(\phi_{c})$ in \eqref{2.13}, we have
\begin{align*}
	S_{c}''(\phi_{c})\Psi_{c}
   &=-cL_{\omega}\Psi_{c}\\
   &=\frac{1}{2}\omega^{-1-\frac{2}{p}}
       \big[-\partial_{xx}(\partial_{\omega} \psi_{\omega})
      +(1-\omega^{2})\partial_{\omega}\psi_{\omega}
      -(p+1)\psi_{\omega}^{p}\partial_{\omega}\psi_{\omega}\big]. 
\end{align*}
This combined with \eqref{3.7} and \eqref{1.4} gives
\begin{align*}
	S_{c}''(\phi_{c})\Psi_{c}
	=\omega^{-\frac{2}{p}}\psi_{\omega} =\phi_{c}. 
\end{align*}
Thus we obtain the desired results.
\end{proof}

%\begin{lem}\label{lem:nega-phi-c}
%	Let $\Psi_{c} $ be the same as Lemma \ref{lem3.3}, then
%	\begin{align}
%		\langle S_{c}''(\phi_{c})\Psi_{c}, \Psi_{c} \rangle  < 0.
%	\end{align}
%\end{lem}
%\begin{proof}
%By \eqref{3.6}, and a direct computation shows that
%\begin{align*}
%		\langle S_{c}''(\phi_{c})\Psi_{c}, \Psi_{c}\rangle 
%	&=-\frac{1}{2}\omega^{1-\frac{2}{p}}
%	\langle 
%	\phi_{c},\partial_{\omega}\psi_{\omega}
%	\rangle \notag\\
%	&=-\frac{1}{4}\omega^{1-\frac{4}{p}}\frac{d}{d\omega}\big\|\psi_{\omega}\big\|^{2}_{L^{2}}. 
%\end{align*}
%Noting that $p>4, $ and by \eqref{1.6}, we have 
%\begin{align*}
%\frac{d}{d\omega}\big\|\psi_{\omega}\big\|^{2}_{L^{2}}
%&=(-2\omega)(\frac{2}{p}-\frac{1}{2})(1-\omega^{2})^{\frac{2}{p}-\frac{3}{2}}\big\|\psi_{0}\big\|^{2}_{L^2}
%>0. 
%\end{align*}
%Thus we complete the proof.
%\end{proof}

\section{Modulation and dynamic of the parameter}
\label{3}
Under the assumption \eqref{second-orth-condition}, in order to obtain a contradiction, we assume that the solitary waves solution is stable, that is: for any $\varepsilon>0$, there exists $\delta>0$ such that when 
 \begin{align*}
 	\big\|u_{0}-\phi_{c}\big\|_{\H1}<\delta, 
 \end{align*}
 we have
 \begin{align}
 	\label{4.1}
 	u \in U_{\varepsilon}(\phi_{c}).
 \end{align}

 \begin{prop}
 	\label{prop3.1}
  Let $c=c_{0}(p). $
 	Suppose that $u(t) \in U_{\varepsilon}(\phi_{c})$ for any $t\in \R$. Then there exist $C^1$ functions 
 \begin{align*}
         y: \R \rightarrow \R, \quad \lambda: \R \rightarrow \R^{+}
     \end{align*}
     such that for 
     \begin{align}
         \label{xi}
         \xi(t)=u(t, \cdot +y(t))-\phi_{\lambda(t)}, 
     \end{align}
     the following orthogonality conditions hold: 
     \begin{align}
         \label{orth}
         \langle \xi, \partial_{x}\phi_{\lambda(t)}\rangle 
         =\langle \xi, \kappa_{\lambda(t)}\rangle 
         =0,
     \end{align}
 where 
 \begin{align} 		\label{kappa_c}
 \kappa_{\lambda}=S_\lambda''(\phi_{\lambda})\Gamma_{\lambda},
 \end{align} 
     and $\xi$ lies in the positive direction of $S_\lambda''(\phi_{\lambda}), $ that is, 
 \begin{align}
 	\label{posi}
 		\langle S_\lambda''(\phi_{\lambda})\xi, \xi\rangle \gtrsim  \|\xi\|^{2}_{H^{1}}.
 \end{align}
%where $	S_{c}''(\phi_{\lambda})f
%=\lambda\partial_{xx} f+(1-\lambda)f+(p+1)\phi_{\lambda}^{p}f.$
 	Furthermore, the following estimate holds:
 	\begin{align}
 		\label{vare}
 		\|\xi\|_{\H1}+|\lambda-c| \lesssim \varepsilon.
 	\end{align}
 \end{prop}
\begin{proof}
%We claim that:
% Suppose that $u(t) \in U_{\varepsilon}(\phi_{c}) $ for any $t\in \R, $ there exists $C^1 $ functions
%$ y: \R \rightarrow \R, \quad \lambda: \R \rightarrow \R^{+}$
%such that for any $t \in \R, \xi(t)=u(t, \cdot +y(t))-\phi_{\lambda(t)},   $ 
% $\xi$ satisfies
%	$$
%		\langle \xi, \partial_{x}\phi_{\lambda(t)}\rangle 
%		=\langle \xi, \phi_{\lambda(t)}\rangle 
%		=0.$$
From Proposition \ref{prop4.1}, we first verify $\kappa_{c}$ satisfying \eqref{Assume-Modulation}.
By \eqref{2.11} and Lemma \ref{lem2.1}, we have
\begin{align*}
	\langle S''_c(\phi_{c})\partial_c\phi_{c}, \partial_c\phi_{c}\rangle=\frac{d}{dc}Q(\phi_{c})\Big|_{c=c_0(p)}=0.
\end{align*}
 Then by Corollary \ref{cor:a.4}, we obtain 
\begin{align*}
	\langle\partial_{c}\phi_{c}, \kappa_{c}\rangle
\neq 0.
\end{align*}
Therefore, there exists $\varepsilon_{0}>0 $ such that for 
$\varepsilon \in (0, \varepsilon_{0}), 
u\in U_{\varepsilon}(\phi_{c}), $
there exists unique $C^{1}$-functions
\begin{align*}
	y: U_{\varepsilon}(\phi_{c}) \rightarrow \R, \quad
	\lambda:  U_{\varepsilon}(\phi_{c}) \rightarrow \R^{+}, 
\end{align*}
such that
\begin{align*}
	\langle \xi, \partial_{x}\phi_{\lambda}\rangle 
	=\langle \xi, \kappa_{\lambda}\rangle 
	=0.
\end{align*}
By \eqref{second-orth-condition}, we have that $\Gamma_{\lambda}$ satisfying \eqref{Assume-Coer}.
From Proposition \ref{prop3.4}, we obtain \eqref{posi}.
Furthermore, 
\begin{align*}
	\left(
	\begin{array}{cccc}
		\partial_{u}\lambda & \partial_{v}\lambda \\ 
		\partial_{u}y & \partial_{v}y
	\end{array}
	\right) 
	=J^{-1}
	\left(
	\begin{array}{cccc}
		\partial_{u}F_{1} & \partial_{v}F_{1} \\ 
		\partial_{u}F_{2} & \partial_{v}F_{2}
	\end{array}
	\right). 
\end{align*}
This implies that
\begin{align*}
	|\lambda-c| \lesssim \big\|u-\phi_{c}\big\|_{\H1}<\varepsilon.
\end{align*}
This finishes the proof of the proposition.
\end{proof}

Some consequences of Proposition \ref{prop3.1} are the follows. 
The first one is the rough estimate on $\dot y$ and $\dot \lambda. $
\begin{cor}
	\label{cor:5.1}
	Let $u$ be the solution of \eqref{1.1} with $u\in U_{\varepsilon}(\phi_{c}), $ where $\varepsilon$ 
	is obtained in Proposition \ref{prop3.1}.
	Let $y, \lambda, \xi$ be the parameters and function obtained in Proposition \ref{prop3.1}, then
	\begin{align*}
		\dot{y} -\lambda
		=O\left(\|\xi\|_{\H1}\right)
	\end{align*}
	and
	\begin{align*}
		\dot{\lambda}=O\left(\|\xi\|_{\H1}\right).
	\end{align*}
\end{cor}

\begin{proof}
	Recall the definition 
	$\xi(t)=u(t, \cdot +y(t))-\phi_{\lambda(t)} $
	in \eqref{xi}, that is 
	\begin{align}
		\label{5.3}
		u(t, x)=\phi_{\lambda}(x-y(t))+\xi(t,x-y(t)).
	\end{align}
	Taking the derivative of 
	\eqref{5.3} with respect to $t, $ we have
	\begin{align}
		\label{ut1}
		u_t=\dot{\xi}+\dot{\lambda}\partial_{\lambda}\phi_{\lambda}-\dot{y}\partial_{x}(\phi_{\lambda}+\xi).
	\end{align} 
	Inserting \eqref{ut1} into equation \eqref{equa-hami}, we get 
	\begin{align}
		\label{ut2}
		\dot{\xi}+\dot{\lambda}\partial_{\lambda}\phi_{\lambda}-\dot{y}\partial_{x}(\phi_{\lambda}+\xi)=JE'(\phi_{\lambda}+\xi),
	\end{align}
where we note that $J=-(1-\partial_{x}^2)^{-1}\partial_{x}. $ Adding $\lambda\partial_{x}(\phi_{\lambda}+\xi)$ to both sides of \eqref{ut2}, we have
\begin{align*}
	\dot{\xi}+\dot{\lambda}\partial_{\lambda}\phi_{\lambda}-(\dot{y}-\lambda)\partial_{x}(\phi_{\lambda}+\xi)
	&=JE'(\phi_{\lambda}+\xi)+\lambda\partial_{x}(\phi_{\lambda}+\xi)
	\notag\\
	&=J\left[E'(\phi_{\lambda}+\xi)-\lambda(1-\partial_{x}^2)(\phi_{\lambda}+\xi)\right].
\end{align*}
	We note that $(1-\partial_{x}^2)(\phi_{\lambda}+\xi)=Q'(\phi_{\lambda}+\xi), $ then the above equality can be rewritten as follows:
	\begin{align}
		\label{dynamic-1}
		\dot{\xi}+\dot{\lambda}\partial_{\lambda}\phi_{\lambda}-(\dot{y}-\lambda)\partial_{x}(\phi_{\lambda}+\xi)
		&=J\left[E'(\phi_{\lambda}+\xi)-\lambda Q'(\phi_{\lambda}+\xi)\right]
		\notag\\
		&=JS'_{\lambda}(\phi_{\lambda}+\xi).
	\end{align} 
Using Taylor's type expansion, we have
\begin{align}
	\label{taylor-S'(phi-la+xi)}
	S'_{\lambda}(\phi_{\lambda}+\xi)
	&=S'_{\lambda}(\phi_{\lambda})+S''_{\lambda}(\phi_{\lambda})\xi+O\big(\xi^2\big)
	\notag\\
	&=S''_{\lambda}(\phi_{\lambda})\xi+O\big(\xi^2\big),
\end{align}
where we used $S'_{\lambda}(\phi_{\lambda})=0. $ 
Inserting \eqref{taylor-S'(phi-la+xi)} into \eqref{dynamic-1}, we have
\begin{align}
	\label{m5.3}
		\dot{\xi}+\dot{\lambda}\partial_{\lambda}\phi_{\lambda}-(\dot{y}-\lambda)\partial_{x}(\phi_{\lambda}+\xi)=JS''_{\lambda}(\phi_{\lambda})\xi+\mathcal{N}_1(\xi),
\end{align}
	where  $\mathcal{N}_1(\xi)$ verifies 
	\begin{align*}
		\langle \mathcal{N}_1(\xi), f\rangle 
		=O\left(\|\xi\|^{2}_{H^{1}}\|f\|_{H^1}\right) ,\quad \mbox{for any } f\in H^1.
	\end{align*}
%	Therefore, from \eqref{m5.3} we get that  
%	\begin{align}
%		\label{5.4}
%		(1-\partial_{x}^2)\dot{\xi}	
%		+\dot{\lambda}(1-\partial_x^{2})\partial_{\lambda}\phi_{\lambda}
%		-(\dot{y}-\lambda)(1-\partial_x^{2})
%		\partial_{x}\left(\phi_{\lambda}+\xi\right)
%		=-\partial_{x}
%		\left(S_{\lambda}''(\phi_{\lambda})\xi\right)
%		+(1-\partial_x^{2})\mathcal{N}_1(\xi).
%	\end{align} 

	Taking inner product by \eqref{m5.3} and $\partial_{x}\phi_{\lambda}, \kappa_{\lambda}$ respectively, by integration-by-parts, we have
	\begin{align}
		\label{5.5}
		\langle 
		\dot{\xi},\partial_{x}\phi_{\lambda}
		\rangle 
		+\dot{\lambda}
		\langle 
		\partial_{\lambda}\phi_{\lambda},\partial_{x}\phi_{\lambda}
		\rangle 
		-(\dot{y}-\lambda)
		\langle 
		\partial_{x}\left(\phi_{\lambda}+\xi\right),\partial_{x}\phi_{\lambda} 
		\rangle 
		&=\langle 
	JS_{\lambda}''(\phi_{\lambda})\xi,\partial_{x}\phi_{\lambda}
		\rangle 
		+O\left(\|\xi\|^{2}_{H^{1}}\right), 
		\\
		\label{dynamic_phi_lambda}
		\langle 
		\dot{\xi}, \kappa_{\lambda}
		\rangle 
		+\dot{\lambda}
		\langle 
		\partial_{\lambda}\phi_{\lambda},  \kappa_{\lambda}
		\rangle 
		-(\dot{y}-\lambda)
		\langle 
		\partial_{x}\left(\phi_{\lambda}+\xi\right),  \kappa_{\lambda}
		\rangle 
		&=\langle 
		JS_{\lambda}''(\phi_{\lambda})\xi,\kappa_{\lambda}
		\rangle 
		+O\left(\|\xi\|^{2}_{H^{1}}\right).
	\end{align}
	By the even property of $\phi_{\lambda}, $ it is known that $\Gamma_{c}=B(c)\left(c^2\Psi_{c}+\frac{c}{2}x\partial_{x}\phi_{c}+c\phi_{c}\right)+D(c)(3x^2\phi_{c}+x^3\partial_{x}\phi_c)$ is an even function. Moreover, we note that $\kappa_{\lambda}$ is also an even function since $S_{\lambda}''(\phi_{\lambda})\Gamma_{\lambda}$ has the same parity as $\Gamma_{\lambda}$. Using orthogonality conditions in \eqref{orth}, we simplify \eqref{5.5} and \eqref{dynamic_phi_lambda} as
	\begin{align}
		\label{sim_parameter1}
		-\dot{\lambda}\langle 
		\xi, \partial_{\lambda}\partial_{x}\phi_{\lambda}\rangle 
		-(\dot{y}-\lambda)\Big(\big\|\partial_{x}\phi_{\lambda}\big\|^{2}_{L^{2}}-\langle\xi, \partial_{xx}\phi_{\lambda}\rangle\Big)
	&	=-\langle 
		\xi, S_{\lambda}''(\phi_{\lambda})\big(J\partial_{x}\phi_{\lambda}\big)
		\rangle 
		+O\left(\|\xi\|^{2}_{H^{1}}\right),
		\\
		\label{sim_parameter2}
		\dot{\lambda}
		\left[
		-\langle \xi, \partial_{\lambda}\kappa_{\lambda}
		\rangle 
		+\langle\partial_{\lambda}\phi_{\lambda}, \kappa_{\lambda}\rangle
		\right] 
		+(\dot{y}-\lambda)\langle \xi, \partial_{x}\kappa_{\lambda}\rangle
	&	=-\langle 
		\xi,S_{\lambda}''(\phi_{\lambda})\big(J\kappa_{\lambda}\big)
		\rangle 
		+O\left(\|\xi\|^{2}_{H^{1}}\right),
	\end{align}
where $\langle\partial_{\lambda}\phi_{\lambda}, \kappa_{\lambda}\rangle$ is a constant denoted by $C(\lambda)$ which only depends on $\lambda.$  
	We denote 
	\begin{align*}
		A=
		\begin{pmatrix}
			&-\langle 
			\xi, \partial_{\lambda}\partial_{x}\phi_{\lambda}
			\rangle 
			&-\big\|\partial_{x}\phi_{\lambda}\big\|^2_{L^2}
			+\langle \xi,\partial_{xx}\phi_{\lambda}\rangle  
			\\
			&-\langle \xi, \partial_{\lambda}\kappa_{\lambda}
			\rangle 
			+C(\lambda)
			&\langle \xi, \partial_{x}\kappa_{\lambda}\rangle
		\end{pmatrix}.
	\end{align*}
	Combining \eqref{sim_parameter1} and \eqref{sim_parameter2}, by a direct computation, we have
	\begin{align}
		\label{5.8}
		\begin{pmatrix}
			\dot{\lambda} \\
			\dot{y}-\lambda
		\end{pmatrix}
		&=A^{-1}
		\begin{pmatrix}
			&-\langle \xi, S_{\lambda}''(\phi_{\lambda})\big(J\partial_{x}\phi_{\lambda}\big)\rangle 
			\\
			&-\langle \xi, S_{\lambda}''(\phi_{\lambda})\big(J\kappa_{\lambda}\big)\rangle
		\end{pmatrix}
		+\begin{pmatrix}
			O\left(\|\xi\|^2_{H^{1}}\right) \\
			O\left(\|\xi\|^2_{H^{1}}\right)
		\end{pmatrix}
	\notag\\
		&=\begin{pmatrix}
			O\left(\|\xi\|_{H^{1}}\right) \\
			O\left(\|\xi\|_{H^{1}}\right)
		\end{pmatrix}.
	\end{align}
	Thus we obtain the desired results.
\end{proof}

The second is a precise estimate on the spatial transform parameter $y(t)$. 
\begin{cor}
 \label{cor:5.2}
 Under the same assumption as in Corollary \ref{cor:5.1}; let $f_\lambda=x^3\phi_{\lambda}, $ then
\begin{align}
\label{estimate_doty-lambda}
\dot{y}-\lambda
&=\frac{1}{B(\lambda)}\langle \xi, S''_{\lambda}(\phi_{\lambda})\partial_{x}f_\lambda\rangle
-\frac{1}{B(\lambda)}\partial_{t}\langle \xi, (1-\partial_{x}^2)f_\lambda\rangle
+O\left(\|\xi\|^{2}_{H^{1}}\right),
\end{align}
where
\begin{align*}
B(\lambda)
=\frac{3}{2}\big\|x\phi_{\lambda}\big\|^{2}_{L^{2}}+\frac{9}{2}\big\|x\partial_{x}\phi_{\lambda}\big\|^{2}_{L^{2}}-3\big\|\phi_{\lambda}\big\|^{2}_{L^{2}}.
\end{align*}
\end{cor}
\begin{proof}

Taking inner product by \eqref{m5.3} and $(1-\partial_{x}^2)f_\lambda, $ by integration-by-parts, we have
\begin{align}
\label{dynamic_doty}
\langle (1-\partial_{x}^{2})\dot{\xi}, f_\lambda\rangle 
+\dot{\lambda}\langle (1-\partial_{x}^{2})\partial_{\lambda}\phi_{\lambda}, f_\lambda\rangle
-(\dot{y}-\lambda)&\langle (1-\partial_{x}^{2})\partial_{x}\phi_{\lambda}, f_\lambda\rangle\notag\\
&=-\langle 
\partial_{x}\left(S_{\lambda}''(\phi_{\lambda})\xi\right), f_\lambda
\rangle 
+O\left(\|\xi\|^{2}_{H^{1}}\right). 
\end{align}
It's worth noting that $f_\lambda=x^3\phi_{\lambda}\in \L2$ since $\phi_{\lambda}$ is exponential decaying.
Now we consider terms in \eqref{dynamic_doty} one by one.
First, from the rough estimate $\dot{\lambda}=O\left(\|\xi\|_{H^{1}}\right)$ in Corollary \ref{cor:5.1}, we have
\begin{align}
\label{first-oder-term}
\langle (1-\partial_{x}^{2})\dot{\xi}, f_\lambda\rangle
&=\partial_{t}\langle \xi, (1-\partial_{x}^{2})(x^3\phi_{\lambda})\rangle
-\dot{\lambda}\langle \xi, (1-\partial_{x}^{2})\left(x^3\partial_{\lambda}\phi_{\lambda}\right)\rangle\notag\\
&=\partial_{t}\langle \xi, (1-\partial_{x}^{2})(x^3\phi_{\lambda})\rangle
+O\left(\|\xi\|^{2}_{H^{1}}\right). 
\end{align} 
The term $\dot{\lambda}\langle (1-\partial_{x}^{2})\partial_{\lambda}\phi_{\lambda}, x^3\phi_{\lambda}\rangle$ vanishes since $\phi_{\lambda}$ is an even function. Then, direct calculation gives that
\begin{align}
\label{cofficient_doty-lam}
-(\dot{y}-\lambda)\langle (1-\partial_{x}^{2})\partial_{x}\phi_{\lambda}, x^3\phi_{\lambda}\rangle
&=(\dot{y}-\lambda)\langle (1-\partial_{x}^{2})\phi_{\lambda}, 3x^2\phi_{\lambda}+x^3\partial_{x}\phi_{\lambda}\rangle\notag\\
&=(\dot{y}-\lambda)\left[\frac{3}{2}\big\|x\phi_{\lambda}\big\|^{2}_{L^{2}}+\frac{9}{2}\big\|x\partial_{x}\phi_{\lambda}\big\|^{2}_{L^{2}}-3\big\|\phi_{\lambda}\big\|^{2}_{L^{2}}\right].
\end{align}
Using the property of $S_\lambda''(\phi_{\lambda})$ in \eqref{2.10}, we have
\begin{align}
		\label{S''partial_xf}
	-\langle \partial_x\left(S''_{\la}(\phi_{\lambda})\xi\right), f_\lambda\rangle
	&=\langle S''_{\la}(\phi_{\lambda})\xi, \partial_xf_\lambda\rangle \notag\\
	&=\langle \xi, S''_{\la}(\phi_{\lambda})\left(3x^2\phi_{\lambda}+x^3\partial_{x}\phi_{\lambda}\right) \rangle.
\end{align}
Combining \eqref{first-oder-term}--\eqref{S''partial_xf}, and thus we complete the proof.
\end{proof}

\section{The localized virial identity}
\label{4}
  The following lemma is the localized virial identity.
% One can see \cite{Y.Liu-2007-strong} for the details of the proof. 
 Let $ y, \lambda, \xi$ be the parameters and function obtained in Corollary \ref{cor:5.1}, $f_\lambda, B(\lambda)$ are the same as Corollary \ref{cor:5.2}.
 Denote
 \begin{align}
 	\label{defi-H(u)}
 	H(u)=-(1-\partial_{x}^2)^{-1}(u+|u|^{p}u).
 \end{align}
From the equation \eqref{equa-hami}, we obtain that $\partial_{x}H(u)=u_{t}.$ 
Inserting  the expression of $u$ in \eqref{5.3} into \eqref{defi-H(u)}, we have
\begin{align*}
	H(u)=-(1-\partial_{x}^2)^{-1}\big(\phi_{\lambda}+\xi+|\phi_{\lambda}+\xi|^{p}(\phi_{\lambda}+\xi)\big).
\end{align*}
Noting that $\phi_{\lambda}$ satisfies 
\begin{align}
	\label{equa-lambda}
	-\lambda\partial_{xx}\phi_{\lambda}+(\lambda-1)\phi_{\lambda}-\phi_{\lambda}^{p+1}=0,
\end{align}
and 
\begin{align}
	\label{S''la}
	S''_{\la}(\phi_{\lambda})f=\lambda\partial_{xx}f+(1-\la)f+(p+1)\phi_{\lambda}^{p}f.
\end{align}
From \eqref{equa-lambda} and \eqref{S''la}, we obtain that
\begin{align}
	\label{H(u)}
	H(u)
	=-\lambda(\phi_{\lambda}+\xi)-(1-\partial_{x}^2)^{-1}\big(S''_{\lambda}(\phi_{\lambda})\xi\big)+(1-\partial_{x}^2)^{-1}\mathcal{N}_2(\xi),
\end{align}
where $\mathcal{N}_2(\xi)$ has the same property as  $\mathcal{N}_1(\xi)$ in \eqref{m5.3} which verifies that 
\begin{align*}
	\langle \mathcal{N}_2(\xi), f\rangle 
	=O\left(\|\xi\|^{2}_{H^{1}}\|f\|_{H^1}\right),\quad \mbox{for any } f\in \H1.
\end{align*}
We also denote that
 \begin{align*}
  &I_{1}(t)=\int_{\R} \varphi(x-y(t))(\frac{1}{2}u^{2}+\frac{1}{p+2}|u|^{p+2}) \dx,\\
  &I_{2}(t)=\frac{D(\lambda)}{B(\lambda)}\langle\xi,  (1-\partial_{x}^{2})(x^3\phi_{\lambda})\rangle,
  \end{align*}
where
\begin{align*}
    D(\lambda)&=-\frac{4p\lambda+4\lambda-3p}{2(p+4)}\big\|\phi_{\lambda}\big\|^2_{L^2},
    \\
	B(\lambda)&=\frac{3}{2}\big\|x\phi_{\lambda}\big\|^2_{L^{2}}+\frac{9}{2}\big\|x\partial_x\phi_{\lambda}\big\|^{2}_{L^2}-3\big\|\phi_{\lambda}\big\|^2_{L^{2}}.
\end{align*}
\begin{lem}
	\label{lem6.1}
	Let $\varphi \in C^{3}(\R), u\in \H1$ be the solution of \eqref{1.1}, then 
\begin{align*}
 I_{1}'(t)
 &=       
 -\dot{y}\int_{\R}
 \varphi'(x-y(t))\left(\frac{1}{2}u^{2}+\frac{1}{p+2}|u|^{p+2}\right)\dx
  +\frac{1}{2}\int_{\R}\varphi'(x-y(t))
 \left[\big(H(u)\big)^{2}-u_{t}^{2}\right]
 \dx,
 \\
 	I_{2}'(t)
 &=\frac{D(\lambda)}{B(\lambda)}\partial_{t}\langle\xi,  (1-\partial_{x}^{2})(x^3\phi_{\lambda})\rangle
 +O\big(\|\xi\|_{H^1}^2\big).
\end{align*}
\end{lem}
\begin{proof}
	First, a direct computation gives that
\begin{align*}
	 I_{1}'(t)=       
	-\dot{y}\int_{\R}
	\varphi'(x-y(t))\left(\frac{1}{2}u^{2}+\frac{1}{p+2}|u|^{p+2}\right)\dx
   +\int_{\R}
   \varphi(x-y(t))\partial_{t}\left(\frac{1}{2}u^{2}+\frac{1}{p+2}|u|^{p+2}\right)\dx.
\end{align*} 
Multiplying \eqref{equa-hami} by $u+|u|^{p}u $ gives:
\begin{align*}
	\partial_{t}\big(\frac{1}{2}u^2+\frac{1}{p+2}|u|^{p+2}\big)
=-(1-\partial_{x}^{2})^{-1}\partial_{x}\big(u+|u|^{p}u \big)\cdot\big(u+|u|^{p}u \big).
\end{align*} 
Further, noting that 
$$u+|u|^{p}u=-(1-\partial_x^2)H(u)=-H(u)+\partial_{t}\partial_{x}u, $$
we get that 
\begin{align*}
	\partial_{t}\big(\frac{1}{2}u^2+\frac{1}{p+2}|u|^{p+2}\big)
&=\partial_{x}H(u)\cdot\big(-H(u)+\partial_{t}\partial_{x}u\big)
\notag\\
&=-\frac{1}{2}\partial_{x}\big[\big(H(u)\big)^{2}-u_t^{2}\big],
\end{align*} 
where we used $\partial_{x}H(u)=u_t$ in the last step.
%
%So we obtain
%\begin{align*}
%  &\int_{\R}
%	\varphi(x-y(t))\partial_{t}\left(\frac{1}{2}u^{2}+\frac{1}{p+2}|u|^{p+2}\right)\dx\\
%=&\int_{\R}\varphi'(x-y(t))
%  \big[\frac{1}{2}(\partial_{x}^{-1}u_t)^{2}+\frac{1}{2}u_t^{2}+(u+|u|^{p}u)\cdot H(u)-\partial_{t}\partial_{x}u \cdot \partial_{x}^{-1}u_t\big]\dx.
%\end{align*}
%Note that
%\begin{align*}
%	&\int_{\R}\varphi'(x-y(t))(u+|u|^{p}u)\cdot H(u) \dx\\
%=&-\int_{\R}\varphi'(x-y(t))(1-\partial_{x}^{2})H(u) \cdot H(u)\dx \\
%=&-\int_{\R}\varphi'(x-y(t))(H(u))^{2} \dx
% +\int_{\R}\varphi'(x-y(t)) \partial_{t}\partial_{x}u \cdot \partial_{x}^{-1}u_t  \dx.
%\end{align*}
Then by integration-by-parts, we obtain that  
\begin{align*}
	I_{1}'(t)=       
	&-\dot{y}\int_{\R}
	\varphi'(x-y(t))\left(\frac{1}{2}u^{2}+\frac{1}{p+2}|u|^{p+2}\right)\dx
	+\frac{1}{2}\int_{\R}\varphi'(x-y(t))
	\left[\big(H(u)\big)^{2}-u_{t}^{2}\right]\dx.
\end{align*}

Second, a direct computation shows that
\begin{align*}
	 I_{2}'(t)=
	&\dot{\lambda}\partial_{\lambda}\Big[\frac{D(\lambda)}{B(\lambda)}\Big] \langle \xi,  (1-\partial_{x}^{2})(x^3\phi_{\lambda})\rangle 
	+\frac{D(\lambda)}{B(\lambda)}\partial_{t}\langle\xi,  (1-\partial_{x}^{2})(x^3\phi_{\lambda})\rangle 
	+O\big(\|\xi\|_{H^1}^2\big) \\
   =&\frac{D(\lambda)}{B(\lambda)}\partial_{t}\langle\xi,  (1-\partial_{x}^{2})(x^3\phi_{\lambda})\rangle 
	+O\big(\|\xi\|_{H^1}^2\big),
\end{align*}
here we used the estimate of $\dot{\lambda}=O\big(\|\xi\|_{H^1}\big) $ in Corollary \ref{cor:5.1} in the last step.
This completes the proof.
\end{proof}

\section{The monotonic functional}
\label{5}
 This section is devoted to prove our main theorem.
\subsection{Virial identities}
 Let $\varphi(x)$ be a smooth cutoff function, where 
 \begin{align}
 	\label{7.1}
 	\varphi(x)=\left\{
 	\begin{aligned}
 		 &x,  \qquad |x|\leqslant R, \\
 		 &0,  \qquad  |x|\geqslant 2R,
 	\end{aligned}
 	 \right.
\end{align}
$0 \leqslant \varphi' \leqslant 1$ for any $x\in\R, $ and $R$ is a large constant decided later.
Moreover, we denote
\begin{align*}
	I(t)=I_{1}(t)+I_{2}(t).
\end{align*}
Then we have the following lemma.

\begin{lem}
	\label{lem7.1}
	Let $R>0, y, \lambda,\xi$ be the parameters and function obtained in Corollary \ref{cor:5.1}. Then
	\begin{subequations}
\begin{align}
	%\label{7.2}
		I_1'(t)
=&-\lambda E(u_{0})
+\frac{1}{2}\lambda^{2}
\left(
 \big\|\phi_{\lambda}\big\|^{2}_{L^{2}}-\big\|\partial_{x}\phi_{\lambda}\big\|^{2}_{L^{2}}
\right)
\label{I1-1}
\\
&+\lambda^{2}\langle\xi, \phi_{\lambda}+\partial_{xx}\phi_{\lambda}\rangle 
 +\lambda \langle \xi, S''_{\lambda}(\phi_{\lambda})\phi_{\lambda} \rangle
-(\dot{y}-\lambda)\Big[E(u_0)+2\lambda
\big\|\partial_{x}\phi_{\lambda}\big\|^{2}_{L^{2}}\Big]\\
% -\frac{1}{B(\lambda)}\partial_{t}\Big(\big\langle \xi, (1-\partial_{x}^2)f\big\rangle\Big) \label{I1-2}\\
&+O\left(\frac{1}{R}+\|\xi\|^{2}_{H^{1}}\right).\notag
\end{align}
\end{subequations}
\end{lem}
\begin{proof}
  From \eqref{lem6.1} and the conversation law of energy, we change the form of $I_{1}'(t)$ as
\begin{align*}
 I_{1}'(t)
=
&-\dot{y}\int_{\R}
   \varphi'(x-y(t))\left(\frac{1}{2}u^{2}+\frac{1}{p+2}|u|^{p+2}\right)\dx+\frac{1}{2}\int_{\R}
 \varphi'(x-y(t))
 \left[\big(H(u)\big)^{2}-u_{t}^{2}\right]\dx
 \notag\\
=
&-\dot{y}E(u_{0})
  +\frac{1}{2}\int_{\R}
  \left[\big(H(u)\big)^{2}-u_{t}^{2}\right]\dx+R(u),
\end{align*}
where
\begin{align}
	\label{7.3}
R(u)
&=\int_{\R}\left[1-\varphi'(x-y(t))\right]
	\left[
	\dot{y}\left(\frac{1}{2}u^{2}+\frac{1}{p+2}|u|^{p+2}\right)
	-\frac{1}{2}\left(H(u)\right)^{2} 
	+\frac{1}{2}u_{t}^{2}
	\right]\dx.
\end{align}
Then we need to consider the terms $\frac{1}{2}\int_{\R}
\left[\big(H(u)\big)^{2}-u_{t}^{2}\right]\dx
$ and $R(u). $

$\bullet$ \textit{ Estimate on 
	$\frac{1}{2}\int_{\R}
	\left[(H(u))^{2}-u_{t}^{2}\right]\dx.$}

Now we consider terms
 $\int_{\R}
 (H(u))^{2}\dx$ and $\int_{\R}u_{t}^{2}\dx $
 respectively.
We recall the expression of $H(u)$ in \eqref{H(u)}, we have
	\begin{align*}
		%\label{-1u}
		\int_{\R}
		(H(u))^{2}\dx
		&=\int_{\R}\left[-\lambda(\phi_{\lambda}+\xi)-(1-\partial_{x}^2)^{-1}\big(S''_{\lambda}(\phi_{\lambda})\xi\big)+(1-\partial_{x}^2)^{-1}\mathcal{N}_2(\xi)\right]^2\dx
		\notag\\
		&=\lambda^2\big\|\phi_{\lambda}+\xi\big\|^{2}_{L^{2}}
		+2\lambda\left\langle \phi_{\lambda}, (1-\partial_{x}^2)^{-1}\big(S''_{\lambda}(\phi_{\lambda})\xi\big)\right\rangle
		\\
		&\quad
		+2\lambda\left\langle \xi, (1-\partial_{x}^2)^{-1}\big(S''_{\lambda}(\phi_{\lambda})\xi\big)\right\rangle
		+\big\|(1-\partial_{x}^2)^{-1}\big(S''_{\lambda}(\phi_{\lambda})\xi\big)\big\|^{2}_{L^{2}}
		+O\big(\|\xi\|_{H^{1}}^2\big).
\end{align*}
Now we estimate the terms above one by one.

(i) The term $\big\|\phi_{\lambda}+\xi\big\|^{2}_{L^{2}}$. 
\begin{align}
	\label{term1}
	\big\|\phi_{\lambda}+\xi\big\|^{2}_{L^{2}}
	=\big\|\phi_{\lambda}\big\|^{2}_{L^{2}}+2\langle \xi, \phi_{\lambda}\rangle+O\big(\|\xi\|_{H^{1}}^2\big).
\end{align}

(ii) The term $\left\langle \phi_{\lambda}, (1-\partial_{x}^2)^{-1}\big(S''_{\lambda}(\phi_{\lambda})\xi\big)\right\rangle$. 
We recall \eqref{sim_parameter1} and insert the rough estimates of $\dot{\lambda}$ and $\dot{y}-\lambda$ obtained in Corollary \ref{cor:5.1} into \eqref{sim_parameter1}, we have that 
\begin{align*}
	-(\dot{y}-\lambda)\big\|\partial_x\phi_{\lambda}\big\|^{2}_{L^{2}}
=\langle \xi, S''_{\lambda}(\phi_{\lambda})\left[(1-\partial_{x}^2)^{-1}\partial_{xx}\phi_{\lambda}\right] \rangle
 +O\big(\|\xi\|_{H^{1}}^2\big).
\end{align*}
We note that $(1-\partial_{x}^2)^{-1}\partial_{xx}\phi_{\lambda}=-\phi_{\lambda}+(1-\partial_{x}^2)^{-1}\phi_{\lambda}, $ then we have
\begin{align}
	\label{1-partial_xx^{-1}}
	\left\langle \phi_{\lambda}, (1-\partial_{x}^2)^{-1}\big(S''_{\lambda}(\phi_{\lambda})\xi\big)\right\rangle
&=\left\langle \xi, S''_{\lambda}(\phi_{\lambda})\big((1-\partial_{x}^2)^{-1}\phi_{\lambda}\big)\right\rangle
\notag\\
&=\left\langle \xi, S''_{\lambda}(\phi_{\lambda})\big((1-\partial_{x}^2)^{-1}\partial_{xx}\phi_{\lambda}\big)\right\rangle
+\left\langle \xi, S''_{\lambda}(\phi_{\lambda})\phi_{\lambda}\right\rangle
\notag\\
&=
-(\dot{y}-\lambda)\big\|\partial_x\phi_{\lambda}\big\|^{2}_{L^{2}} 
+\langle \xi, S''_{\lambda}(\phi_{\lambda})\phi_{\lambda}\rangle
+O\big(\|\xi\|_{H^{1}}^2\big).
\end{align}

(iii) The term $\big\|(1-\partial_{x}^2)^{-1}\big(S''_{\lambda}(\phi_{\lambda})\xi\big)\big\|^{2}_{L^{2}}$.
Using the expression of $S''_{\lambda}(\phi_{\lambda})$ in \eqref{S''la}, we have
\begin{align*}
	(1-\partial_{x}^2)^{-1}\big(S''_{\lambda}(\phi_{\lambda})\xi\big)
	=-\lambda\xi+(1-\partial_{x}^2)^{-1}\left[\xi+(p+1)\phi_{\lambda}^{p}\xi\right]
\end{align*}
Thus by Young's inequality, we have
\begin{align}
	\label{1-partial_xx-S''xi}
	\big\|(1-\partial_{x}^2)^{-1}\big(S''_{\lambda}(\phi_{\lambda})\xi\big)\big\|^{2}_{L^{2}}
	&\lesssim
	\big\|\xi\big\|^{2}_{L^{2}}
	+\big\|(1-\partial_{x}^2)^{-1}\left[\xi+(p+1)\phi_{\lambda}^{p}\xi\right]\big\|^{2}_{L^{2}}
	\notag\\
	&\lesssim
	\big\|\xi\big\|^{2}_{L^{2}}+\big\|\xi+(p+1)\phi_{\lambda}^{p}\xi\big\|^{2}_{L^{2}}
	\notag\\
	&\lesssim
	\big\|\xi\big\|^{2}_{L^{2}},
\end{align}
where we used $\phi_{\lambda}\in L^{\infty}(\R)$ in the last step.

(iv) The term $\left\langle \xi, (1-\partial_{x}^2)^{-1}\big(S''_{\lambda}(\phi_{\lambda})\xi\big)\right\rangle$. 
By H\"older's inequality and \eqref{1-partial_xx-S''xi}, we obtain that 
\begin{align}
	\label{xi-1-partial_xx-S''xi}
	\left\langle \xi, (1-\partial_{x}^2)^{-1}\big(S''_{\lambda}(\phi_{\lambda})\xi\big)\right\rangle
		&\lesssim
	\big\|\xi\big\|_{L^{2}}
	\cdot
	\big\|(1-\partial_{x}^2)^{-1}\big(S''_{\lambda}(\phi_{\lambda})\xi\big)\big\|_{L^{2}}
	\notag\\
	&\lesssim
	\big\|\xi\big\|^2_{L^{2}}.
\end{align}

	Collecting all the estimates above,  we obtain that 
\begin{align}
	\label{-1ut}
     \int_{\R}(H(u))^{2}\dx
&=\lambda^{2}\big\|\phi_{\lambda}\big\|^{2}_{L^{2}}
   +2\lambda^{2}\langle \xi, \phi_{\lambda}\rangle
   +2\lambda \langle \xi, S''_{\lambda}(\phi_{\lambda})\phi_{\lambda}\rangle
   -2\lambda(\dot{y}-\lambda)\big\|\partial_{x}\phi_{\lambda}\big\|^{2}_{L^{2}}
  +O\left(\|\xi\|^{2}_{H^{1}}\right). 
\end{align}

Arguing similarly, taking the derivative of \eqref{H(u)} with respect to $x$, we have that
\begin{align}
	\label{u_t}
	u_t
	&=-(1-\partial_{x}^2)^{-1}\partial_{x}(u+|u|^{p}u)
	\notag\\
	&=-\lambda\partial_x(\phi_{\lambda}+\xi)-(1-\partial_{x}^2)^{-1}\partial_{x}\left[ S''_{\lambda}(\phi_{\lambda})\xi\right]
	+(1-\partial_{x}^2)^{-1}\partial_{x}\mathcal{N}_2(\xi).
\end{align}
Repeating the process above, we obtain that 
\begin{align}
	\label{int_u_t^2}
	\int_{\R}u_{t}^{2}\dx
	=\lambda^{2}\big\|\partial_{x}\phi_{\lambda}\big\|^{2}_{L^{2}}
	-2\lambda^2\langle
	\xi, \partial_{xx}\phi_{\lambda} \rangle 
	+2\lambda(\dot{y}-\lambda)\big\|\partial_{x}\phi_{\lambda}\big\|^{2}_{L^{2}}
     +O\left(\|\xi\|_{H^{1}}^2\right). 
\end{align} 
From \eqref{-1ut} and \eqref{int_u_t^2}, we have
\begin{align}
\frac{1}{2}\int_{\R}
\left[(H(u))^{2}-u_{t}^{2}\right]\dx
	&=\frac{1}{2}\lambda^{2}
	\left(
	\big\|\phi_{\lambda}\big\|^{2}_{L^{2}}
	-\big\|\partial_{x}\phi_{\lambda}\big\|^{2}_{L^{2}}
	\right)
	+\lambda^{2}\langle\xi, \phi_{\lambda}+\partial_{xx}\phi_{\lambda}\rangle \notag\\
   &\quad
    +\lambda \langle \xi, S''_{\lambda}(\phi_{\lambda})\phi_{\lambda} \rangle
    -2\lambda(\dot{y}-\lambda)\big\|\partial_{x}\phi_{\lambda}\big\|^{2}_{L^{2}}
    +O\left(\|\xi\|^{2}_{H^{1}}\right).
     \label{est:main-part1}
\end{align}

$\bullet$\textit{Estimate on $R(u)$. } \\

Using the definition of the cutoff function $\varphi$ in \eqref{7.1}, we have
\begin{align*}
   \big|R(u)\big|
&\leqslant
 \frac{1}{2}\Bigg|\int_{|x-y(t)|>R}
   \left[1-\varphi'(x-y(t))\right]
   \left[\big(H(u)\big)^{2}-u_{t}^{2}\right]
   \dx\Bigg|\\
&\quad
   +|\dot{y}|
   \Bigg|\int_{|x-y(t)|>R}
   \left[1-\varphi'(x-y(t))\right]
  \left[\frac{1}{2}u^{2}+\frac{1}{p+2}|u|^{p+2}\right]\dx
  \\
&\lesssim 
  \int_{|x|>R}
  \left[u^{2}+|u|^{p+2}+\big|H(u)\big|^{2}+u_{t}^{2}
  \right]\dx.
\end{align*}
By H\"older's inequality, Corollary \ref{cor:5.1}, \eqref{1.3}, \eqref{S''la}, \eqref{-1ut} and \eqref{int_u_t^2}, we have
\begin{align*}
	\big|R(u)\big|
&\lesssim 
 \int_{|x|>R}
 \left[
 (\phi_{\lambda}+\xi)^{2}+|\phi_{\lambda}+\xi|^{p+2}
 +|\phi_{\lambda}|^{2}+|\partial_{x}\phi_{\lambda}|^{2}]
 \right]\dx\\
&\qquad  +\Big|\int_{|x|>R}
   \xi \cdot (\phi_{\lambda}+\partial_{xx}\phi_{\lambda}) \dx\Big|
  +O(\|\xi\|^{2}_{H^{1}}).
\end{align*}
 Further, using the property of exponential decay of $ \phi_{\lambda}, \partial_{xx}\phi_{\lambda} $ we have
\begin{align*}
    \int_{|x|>R} |\phi_{\lambda}|^{2}+|\partial_{x}\phi_{\lambda}|^{2} \dx
\leqslant
    C\int_{|x|>R}e^{-C|x|}\dx
\leqslant
    \frac{C}{R}. 
\end{align*} 
Then Young's inequality gives that
\begin{align}
\label{7.8}
   \int_{|x|>R}\xi \cdot (\phi_{\lambda}+\partial_{xx}\phi_{\lambda})\dx
\lesssim
\frac{1}{R}+\|\xi\|^{2}_{L^{2}},\\
	\label{7.9}
	\int_{|x|>R} (\phi_{\lambda}+\xi)^{2} \dx
\lesssim
    \frac{1}{R}+\|\xi\|^{2}_{L^{2}},\\
    \label{7.10}
    \int_{|x|>R} (\phi_{\lambda}+\xi)^{p+2} \dx
\lesssim
    \frac{1}{R}+\|\xi\|^{2}_{H^{1}}.
\end{align}
Therefore, we combine \eqref{7.8}--\eqref{7.10} to obtain
\begin{align} \label{est:remainder}
	\big|R(u)\big|
\leqslant
  C\left(\frac{1}{R}+\|\xi\|^{2}_{H^{1}}\right).
\end{align}

Now inserting the estimates in \eqref{est:main-part1} and \eqref{est:remainder} into \eqref{7.3}, we give the desired estimate  and 
%This implies
%\begin{align*}
%	\big|R(u)\big|=
%	O\left(\frac{1}{R}+\|\xi\|^{2}_{H^{1}}\right).
%\end{align*}
thus  complete the proof of the lemma.
\end{proof}

\subsection{Structure of $I'(t)$}
Denote
\begin{align}
	\label{7.11}
   \beta(u_{0})
&=-\lambda\Big[E(u_{0})-E(\phi_{c})\Big], 
    \\
    \label{7.12}
	\gamma(\lambda)
&=-\lambda E(\phi_{c})
  +\frac{1}{2}\lambda^{2}
  \left(\big\|\phi_{\lambda}\big\|^{2}_{L^{2}}-\big\|\partial_{x}\phi_{\lambda}\big\|^{2}_{L^{2}}\right),
  \\
 \label{7.13}
  \tilde{R}(u)
&=R(u)+\lambda^{2}\langle\xi, \phi_{\lambda}+\partial_{xx}\phi_{\lambda}\rangle
+\lambda \langle \xi, S''_{\lambda}(\phi_{\lambda})\phi_{\lambda}\rangle
-(\dot{y}-\lambda)\left[E(u_0)+2\lambda\big\|\partial_{x}\phi_{\lambda}\big\|^{2}_{L^{2}}\right]
\notag\\
&\quad
+\frac{D(\lambda)}{B(\lambda)}\partial_{t}\langle \xi, (1-\partial_{x}^2)(x^3\phi_{\lambda})\rangle 
+O\left(\|\xi\|^{2}_{H^{1}}\right),
\end{align}
where 
\begin{align*}
	D(\lambda)=-\frac{4p\lambda+4\lambda-3p}{2(p+4)}\big\|\phi_{\lambda}\big\|^{2}_{L^{2}},
	\quad
	B(\lambda)=\frac{3}{2}\big\|x\phi_{\lambda}\big\|^2_{L^{2}}+\frac{9}{2}\big\|x\partial_x\phi_{\lambda}\big\|^{2}_{L^2}-3\big\|\phi_{\lambda}\big\|^2_{L^{2}}.
\end{align*}

\begin{lem}
	\label{lem7.3}
	It holds that
\begin{align*}
	I'(t)=\beta(u_{0})+\gamma(\lambda)+\tilde{R}(u).
\end{align*}	
\begin{proof}
It follows from Lemma \ref{lem6.1} and Lemma \ref{lem7.1} directly.
\end{proof}

\begin{lem}
	\label{lem7.4}
	We estimate $\tilde{R}(u)$ as follows:
	\begin{align}
		\label{7.14}
		\tilde{R}(u)
		=O(\frac{1}{R}+\|\xi\|^{2}_{H^{1}}).
	\end{align}
\end{lem}
\begin{proof}
Recall the definition of $\tilde{R}(u)$ in \eqref{7.13}:
\begin{align*}
	\tilde{R}(u)
	&=
	R(u)
	+\lambda^{2}\langle\xi, \phi_{\lambda}+\partial_{xx}\phi_{\lambda}\rangle
	+\lambda \langle \xi, S''_{\lambda}(\phi_{\lambda})\phi_{\lambda}\rangle
	-(\dot{y}-\lambda)\left[
	E(u_{0})
	+2\lambda\big\|\partial_{x}\phi_{\lambda}\big\|^{2}_{L^{2}}
	\right]
	\\
	&\quad
	+\frac{D(\lambda)}{B(\lambda)}\partial_{t}\langle \xi, (1-\partial_{x}^2)(x^3\phi_{\lambda})\rangle 
	+O\left(\|\xi\|^{2}_{H^{1}}\right).
\end{align*}
First, from Lemma \ref{lem3.3}, we have already known that $\phi_{\lambda}$ and $\partial_{xx}\phi_{\lambda}$ both have pre-image with respect to $S''_{\lambda}(\phi_{\lambda}), $ then we have
\begin{align}
	\label{tidel-R-part1}
	\langle\xi, \phi_{\lambda}+\partial_{xx}\phi_{\lambda}\rangle=\langle\xi, S''_{\lambda}(\phi_{\lambda})\big(\Psi_{\lambda}+\frac{1}{2\lambda}x\partial_{x}\phi_{\lambda}\big)\rangle.
\end{align}
By \eqref{5.3} and Taylor's type expansion, we have
\begin{align*}
	E(u_{0})=E(u)
	&=E(\phi_{\lambda})+\langle E'(\phi_{\lambda}),\xi\rangle 
	+O(\|\xi\|^{2}_{H^{1}})\\
	&=E(\phi_{\lambda})+\langle \xi, \phi_{\lambda}+\phi_{\lambda}^{p+1}
	\rangle 
	+O\left(\|\xi\|^{2}_{H^{1}}\right).
\end{align*} 
From Lemma \ref{lem2.1} and \eqref{2.1}, we have
\begin{align*}
	E(\phi_{\lambda})&=\frac{4\lambda+p}{2(p+4)}\big\|\phi_{\lambda}\big\|^{2}_{L^{2}};
	\\
	E(\phi_{\lambda})+2\lambda\big\|\partial_{x}\phi_{\lambda}\big\|^{2}_{L^{2}}
	&=\frac{4p\lambda+4\lambda-3p}{2(p+4)}\big\|\phi_{\lambda}\big\|^{2}_{L^{2}}.
\end{align*}
Combining the rough estimate of $\dot{y}-\lambda=O\left(\|\xi\|_{H^1}\right)$ in Corollary \ref{cor:5.1} and the precise estimate of $\dot{y}-\lambda$ in \eqref{estimate_doty-lambda}, we have 
\begin{align}
	\label{7.16}
	&-(\dot{y}-\lambda)\left[
	E(u_{0})
	+2\lambda\big\|\partial_{x}\phi_{\lambda}\big\|^{2}_{L^{2}}
	\right]
	\notag\\
	=&
	-(\dot{y}-\lambda)\left[
	E(\phi_{\lambda})
	+2\lambda\big\|\partial_{x}\phi_{\lambda}\big\|^{2}_{L^{2}}
	+\langle \xi, \phi_{\lambda}+\phi_{\lambda}^{p+1}
	\rangle+O\left(\|\xi\|^{2}_{H^{1}}\right)
	\right]\notag\\
	=&-\frac{4p\lambda+4\lambda-3p}{2(p+4)B(\lambda)}\big\|\phi_{\lambda}\big\|^{2}_{L^{2}}
	\cdot
	\langle\xi, S''_{\lambda}(\phi_{\lambda})(3x^2\phi_{\lambda}+x^3\partial_{x}\phi_{\lambda})\rangle
	\notag\\
	&\quad
	+\frac{4p\lambda+4\lambda-3p}{2(p+4)B(\lambda)}\big\|\phi_{\lambda}\big\|^{2}_{L^{2}}
	\cdot
	\partial_{t}\langle \xi, (1-\partial_{x}^{2})\big(x^3\phi_{\lambda}\big)\rangle
	+O\left(\|\xi\|^{2}_{H^{1}}\right).
\end{align}
Inserting \eqref{tidel-R-part1} and \eqref{7.16} into the expression of $\tilde{R}(u)$, we have
\begin{align*}
	\tilde{R}(u)
	=R(u)+
	\langle\xi, S''_{\lambda}(\phi_{\lambda})
	\big[\lambda^2\Psi_{\lambda}+\frac{\lambda}{2}x\partial_{x}\phi_{\lambda}+\lambda\phi_{\lambda}+\frac{D(\lambda)}{B(\lambda)}(3x^2\phi_{\lambda}+x^3\partial_{x}\phi_{\lambda})\big]
	\rangle.
\end{align*}
We note that $S''_{\lambda}(\phi_{\lambda})
\big[B(\lambda)\big(\lambda^2\Psi_{\lambda}+\frac{\lambda}{2}x\partial_{x}\phi_{\lambda}+\lambda\phi_{\lambda}\big)+D(\lambda)(3x^2\phi_{\lambda}+x^3\partial_{x}\phi_{\lambda})\big]=\kappa_{\lambda}$. 
By the second orthogonality condition \eqref{orth} in Proposition \ref{prop3.1}, we complete the proof.
\end{proof}
\end{lem}

\subsubsection{Lower bound of $\beta(u_0)$}
\begin{lem}
	\label{lem7.5}
	Let $u_{0}=(1-a)\phi_{c}$ for some small positive constant $a$.
	Then there exist a constant $C_{1}>0$, such that
	$$
	\beta(u_{0})\geq C_{1}a.
	$$
\end{lem}
\begin{proof}
	Recall the definition of $\beta(u_{0})$ in \eqref{7.11}:
	\begin{align*}
		\beta(u_{0})
		=-\lambda\Big[E(u_{0})-E(\phi_{c})\Big].
	\end{align*}
	Using the expression in \eqref{2.6} and Taylor's type expansion, we have
	\begin{align}
		\label{7.17}
		E(u_{0})-E(\phi_{c})
		&=\langle E'(\phi_{c}), u_{0}-\phi_{c}\rangle
		+O(\big\|u_{0}-\phi_{c}\big\|^{2}_{H^{1}}) \notag\\
		&=-a\int_{\R}
		\left(\phi_{c}+\phi_{c}^{p+1}\right) \cdot \phi_{c}\dx
		+O(a^{2})\notag\\
		&=-a\left[
		\frac{2(p+2)c}{p+4}-\frac{p}{p+4}\right]\big\|\phi_{c}\big\|^{2}_{L^{2}} +O(a^{2}).
	\end{align}
	Then we put \eqref{7.17} into the expression of $\beta(u_{0}), $
	\begin{align*}
		\beta(u_{0})
		&=-\lambda\Big[E(u_{0})-E(\phi_{c})\Big] \notag\\
		&=a\lambda\left[
		\frac{2(p+2)c}{p+4}-\frac{p}{p+4}\right]
		\big\|\phi_{c}\big\|^{2}_{L^{2}}+O(a^{2})\notag\\
		&=ac\left[
		\frac{2(p+2)c}{p+4}-\frac{p}{p+4}\right]
		\big\|\phi_{c}\big\|^{2}_{L^{2}}+O(a|\lambda-c|)+O(a^{2}).
	\end{align*}
	Note that $c>1,$
	and choosing $a$ and $\varepsilon_{0}$ small enough, where $\varepsilon_{0}$ is the constant in Proposition \ref{prop3.1}, and by \eqref{vare}, we obtain the conclusion of this lemma.
\end{proof}

\subsubsection{Lower bound of $\gamma(\lambda)$}
\begin{lem}
	\label{gamma}
	There exists a positive constant $C_2$ such that 
$$\gamma(\lambda)\geq C_{2}(\lambda-c)^{2}+o\left((\lambda-c)^{2}\right).
$$
\end{lem}
\begin{proof}
Recall the definition of $\gamma(\lambda) $ from  \eqref{7.12}:
\begin{align*}
 	\gamma(\lambda)
&=-\lambda E(\phi_{c})
 +\frac{1}{2}\lambda^{2}
 \left(\big\|\phi_{\lambda}\big\|^{2}_{L^{2}}-\big\|\partial_{x}\phi_{\lambda}\big\|^{2}_{L^{2}}\right).
\end{align*} 
	We claim that 
\begin{align}
	\label{7.18}
   \gamma(c)=0, \quad \gamma'(c)=0, \quad \gamma''(c)>0.
\end{align}
We prove the claim by the following three steps.

\textit{Step 1. $ \gamma(c)=0. $}

From \eqref{2.1}, we have
\begin{align*}
	E(\phi_{c})
=\left(\frac{1}{2}+\frac{2(c-1)}{p+4}\right)\big\|\phi_{c}\big\|^{2}_{L^{2}};
  \\
  \big\|\phi_{\lambda}\big\|^{2}_{L^{2}}-\big\|\partial_{x}\phi_{\lambda}\big\|^{2}_{L^{2}}
=\frac{4\lambda+p}{(p+4)\lambda}\big\|\phi_{\lambda}\big\|^{2}_{L^{2}}.
\end{align*}
So, we have
\begin{align}
	\label{7.19}
	\gamma(\lambda)
=-\lambda\left(
  \frac{1}{2}+\frac{2(c-1)}{p+4}\right)
  \big\|\phi_{c}\big\|^{2}_{L^{2}}
  +\frac{4\lambda^{2}+p\lambda}{2(p+4)}
  \big\|\phi_{\lambda}\big\|^{2}_{L^{2}}.
\end{align}
A direct computation gives
\begin{align*}
\gamma(c)
&=-c\left(
   \frac{1}{2}+\frac{2(c-1)}{p+4}\right)
   \big\|\phi_{c}\big\|^{2}_{L^{2}}
  +\frac{4c^{2}+pc}{2(p+4)}
  \big\|\phi_{c}\big\|^{2}_{L^{2}}
=0.
\end{align*}

\textit{Step  2. $ \gamma'(c)=0. $}

Using the expression of $\gamma(\lambda)$ in \eqref{7.19}, we have
\begin{align}
	\label{7.20}
	\gamma'(\lambda)
&=-\left(
\frac{1}{2}+\frac{2(c-1)}{p+4}\right)
\big\|\phi_{c}\big\|^{2}_{L^{2}}
  +\frac{8\lambda+p}{2(p+4)}\big\|\phi_{\lambda}\big\|^{2}_{L^{2}}\
+\frac{4\lambda^{2}+p\lambda}{2(p+4)}\partial_{\lambda}\big\|\phi_{\lambda}\big\|^{2}_{L^{2}}.
\end{align}
By \eqref{2.4}, we have 
\begin{align}
	\label{dla}
	\partial_{\lambda}\big(\big\|\phi_{\lambda}\big\|^{2}_{L^{2}}\big)
=\frac{4\lambda-p}{2p\lambda(\lambda-1)}\big\|\phi_{\lambda}\big\|^{2}_{L^{2}}, 
\end{align}
so we have
\begin{align*}
\gamma'(c)
&=\left[
 -\left(\frac{1}{2}+\frac{2(c-1)}{p+4}\right)
 +\frac{8c+p}{2(p+4)}
 +\frac{(4c+p)(4c-p)}{4p(p+4)(c-1)}
 \right]
 \big\|\phi_{c}\big\|^{2}_{L^{2}}
=0.
\end{align*}

\textit{Step 3. $ \gamma''(c)>0. $}

From the expression of $\gamma(\lambda)$ in \eqref{7.19}, we have
\begin{align*}
	\gamma''(\lambda)
	=\frac{8\lambda+p}{p+4}\partial_{\lambda}\big\|\phi_{\lambda}\big\|^{2}_{L^{2}}
	+\frac{4}{p+4}\big\|\phi_{\lambda}\big\|^{2}_{L^{2}}
	+\frac{4\lambda^{2}+p\lambda}{2(p+4)}\partial_{\lambda\lambda}\big\|\phi_{\lambda}\big\|^{2}_{L^{2}}.
\end{align*}
From \eqref{dla}, we have
\begin{align*}
	\partial_{\lambda\lambda}\big\|\phi_{\lambda}\big\|^{2}_{L^{2}}
	&=\frac{-4\lambda^2+2p\lambda-p}{2p\lambda^2(\lambda-1)^2}\big\|\phi_{\lambda}\big\|^{2}_{L^{2}}+\left[\frac{4\lambda-p}{2p\lambda(\lambda-1)}\right]^2\big\|\phi_{\lambda}\big\|^{2}_{L^{2}}\\
	&=\frac{8(2-p)\lambda^{2}+4p(p-2)\lambda-p^{2}}{4p^{2}\lambda^{2}(\lambda-1)^{2}}\big\|\phi_{\lambda}\big\|^{2}_{L^{2}}.
\end{align*}
So we get
\begin{align*}
	\gamma''(c)
	=\frac{p-4c}{2p(c-1)^{2}}\big\|\phi_{c}\big\|^{2}_{L^{2}},
\end{align*}
noting that $p-4c=\frac{p}{p+2}\left[p-2\sqrt{2+\frac{1}{2}p}\right]>0$ when $p>4,$
then $	\gamma''(c)>0. $ This proves the claim \eqref{7.18}.

Using \eqref{7.18} and Taylor's type expansion, we get
\begin{align*}
	\gamma(\lambda)
	&=\gamma(c)+\gamma'(c)(\lambda-c)+\frac{1}{2}\gamma''(c)(\lambda-c)^{2}+o\left((\lambda-c)^{2}\right)\\
	&\geq
	C_{2}(\lambda-c)^{2}+o\left((\lambda-c)^{2}\right),
\end{align*}
where
$C_{2}=\frac{1}{2}\gamma''(c)>0.$
Thus we obtain the conclusion of this lemma.
\end{proof}

Hence, combining Lemmas \ref{lem7.3}--\ref{lem7.5}, and \eqref{7.14}, we have
\begin{align}
	\label{7.21}
	I'(t)
	\geq
	C_{1}a+C_{2}(\lambda-c)^{2}
	+O\left(\frac{1}{R}+\|\xi\|^{2}_{H^{1}}\right)
	+o\left((\lambda-c)^{2}\right).
\end{align}

\subsubsection{Upper bound of $\|\xi\|_{H^{1}} $}
\begin{lem}
	\label{lem7.6}
	Let $\xi$ be defined in \eqref{xi}, then for any $t \in \R, $
\begin{align*}
	\|\xi\|^{2}_{H^{1}}
\lesssim
 O\left(a|\lambda-c|+a^{2}\right)
 +o\left((\lambda-c)^{2}\right).
\end{align*}
\end{lem}
\begin{proof}
	First, since $ u=\left(\phi_{\lambda}+\xi\right)(x-y)$ in \eqref{5.3}, by Taylor's type extension and $S'_{c}(\phi_{c})=0$, we have
\begin{align*}
	S_{\lambda}(u)-S_{\lambda}(\phi_{\lambda})
&=\langle
S'_{\lambda}(\phi_{\lambda}), \xi\rangle 
+\frac{1}{2}\langle
S''_{\lambda}(\phi_{\lambda})\xi, \xi\rangle 
+o\left(\|\xi\|^{2}_{H^{1}}\right)
\\
&=\frac{1}{2}\langle
S''_{\lambda}(\phi_{\lambda})\xi, \xi\rangle 
+o\left(\|\xi\|^{2}_{H^{1}}\right).
\end{align*} 
Using $S'_{c}(\phi_{c})=0$ and Taylor's type extension, we have
\begin{align*}
  S_{\lambda}(u)-S_{\lambda}(\phi_{\lambda})
=\frac{1}{2}\langle
  S''_{\lambda}(\phi_{\lambda})\xi, \xi\rangle 
  +o\left(\|\xi\|^{2}_{H^{1}}\right).
\end{align*}
Then by Proposition \ref{prop3.1}, we get
\begin{align*}
	S_{\lambda}(u)-S_{\lambda}(\phi_{\lambda})
\gtrsim
 \|\xi\|^{2}_{H^{1}}.
\end{align*}
Second, note that
\begin{align*}
	S_{\lambda}(u)-S_{\lambda}(\phi_{\lambda})
=S_{\lambda}(u_{0})-S_{\lambda}(\phi_{c})+S_{\lambda}(\phi_{c})-S_{\lambda}(\phi_{\lambda}),
\end{align*}
and the expression of $S_c$ in \eqref{2.5} gives that 
\begin{align*}
	S_{\lambda}(u_{0})-S_{\lambda}(\phi_{c})
=S_{c}(u_{0})-S_{c}(\phi_{c})
 -(\lambda-c)\big[Q(u_{0})-Q(\phi_{c})].
\end{align*}
Using the Taylor's type expansion, by $S'_{c}(\phi_{c})=0, $ \eqref{2.7} and \eqref{2.1}, we have
\begin{align*}
	S_{c}(u_{0})-S_{c}(\phi_{c})
&=\langle S'_{c}(\phi_{c}),u_{0}-\phi_{c}\rangle 
 +O\left(\big\|u_{0}-\phi_{c}\big\|^{2}_{H^{1}}\right)\\
&=O(a^{2});
 \\
 Q(u_{0})-Q(\phi_{c})
&=\langle Q'(\phi_{c}),u_{0}-\phi_{c}\rangle 
  +O\left(\big\|u_{0}-\phi_{c}\big\|^{2}_{H^{1}}\right)\\
&=-a\left[1+\frac{p(c-1)}{c(p+4)}\right]\big\|\phi_{c}\big\|^{2}_{L^{2}}+O(a^{2}).
\end{align*}
So we obtain
\begin{align*}
	S_{\lambda}(u_{0})-S_{\lambda}(\phi_{c})
&=(\lambda-c)a\left[1+\frac{p(c-1)}{c(p+4)}\right]\big\|\phi_{c}\big\|^{2}_{L^{2}}+O(a^{2})\\
&=O\left(a^{2}+a|\lambda-c|\right).
\end{align*}
Moreover, by Corollary \ref{cor2.2}, we have
\begin{align*}
	S_{\lambda}(\phi_{c})-S_{\lambda}(\phi_{\lambda})
=o\left((\lambda-c)^{2}\right).
\end{align*}
Finally, we get the desired result
\begin{align*}
	\|\xi\|^{2}_{H^{1}}
&\lesssim
  S_{\lambda}(u)-S_{\lambda}(\phi_{\lambda})
=S_{\lambda}(u_{0})-S_{\lambda}(\phi_{c})+S_{\lambda}(\phi_{c})-S_{\lambda}(\phi_{\lambda})\\
&=O\left(a^{2}+a|\lambda-c|\right)
 +o\left((\lambda-c)^{2}\right).
\end{align*}
This completes the proof.
\end{proof}

\subsection{Proof of Theorem 1.2}
  As in the discussion above, we assume that 
  $u\in U_{\varepsilon}(\phi_{c}), $ and thus $|\lambda-c|\lesssim \varepsilon. $ 
  We note that from the definition of $I(t) $ and \eqref{2.1} we have the uniform boundedness of $I(t):$
 \begin{align}
 	\label{7.22}
 	\sup_{t \in \R} I(t)
 \lesssim
  R(\big\|\phi_{c}\big\|^{2}_{L^{2}}+1).
 \end{align}
Now we estimate $I'(t). $
From \eqref{7.21} and Lemma \ref{lem7.6}, we have
\begin{align*}
	I'(t)
	&\geq
	C_{1}a+C_{2}(\lambda-c)^{2}
	+O\left(\frac{1}{R}+\|\xi\|^{2}_{H^{1}}\right)
	+o\left((\lambda-c)^{2}\right)\\
	&\geq
	C_{1}a+C_{2}(\lambda-c)^{2}
	+O\left(a^{2}+a|\lambda-c|\right)
	+o\left((\lambda-c)^{2}\right)
	+O\left(\frac{1}{R}\right).
\end{align*}
By \eqref{vare}, choosing $R$ satisfying
$\frac{1}{R}\leqslant a^{2}$ and $ \varepsilon, a_0 $ small enough, we obtain that for any $a\in (0, a_{0}), $
\begin{align*}
	I'(t)
	&\geq
	C_{1}a+C_{2}(\lambda-c)^{2}
	+O\left(a^{2}+a|\lambda-c|\right)
	+o\left((\lambda-c)^{2}\right)\\
	&\geq
	\frac{1}{2}C_{1}a+\frac{1}{2}C_{2}(\lambda-c)^{2}.
\end{align*}
This implies $I(t)\rightarrow +\infty $ when $t\rightarrow +\infty, $ which is contradicted with \eqref{7.22}. Hence we prove the instability of solitary wave solution $\phi_{c}(x-ct)$ and thus give the proof of Theorem \ref{main:thm}.

\appendix
\section{}
\label{appendix}

\subsection{Spectrum  of $S_{c}''(\phi_{c}) $}

First, we study the kernel of $S_{c}''(\phi_{c}) $ in the following lemma. The proof is standard, and it is a consequence of the result from \cite{weinstein-1985-modulational}. 
\begin{lem}
	\label{lem3.1}
	The kernel of $S_{c}''(\phi_{c}) $ satisfies that
\begin{align*}
	 \ker S_{c}''(\phi_{c})
	=\{\alpha \partial_x \phi_{c}:\alpha \in \R \}. 
\end{align*}
\end{lem}
\begin{proof}
First, we need to show the relationship $``\supset". $ For any 
$ f \in \{{\alpha \partial_x \phi_{c}:\alpha \in \R}\}, $
using \eqref{1.3}, we have
\begin{align}
	\label{3.1}
	S_{c}''(\phi_{c})f
	&=S_{c}''(\phi_{c})(\alpha \partial_x \phi_{c})\notag\\
	&=\alpha \partial_x (c\partial_{xx} \phi_{c}+(1-c)\partial_{c} \phi_{c}+\phi_{c}^{p+1})\notag \\
	&=0
\end{align}
Then \eqref{3.1} implies that $f$ is in the kernel of 
$S_{c}''(\phi_{c}), $ and we have the conclusion
\begin{align*}
	\ker S_{c}''(\phi_{c})
	\supset \{ \alpha \partial_x \phi_{c}:\alpha \in \R \}.
\end{align*}
Second, we prove the reverse relationship $``\subset". $
By the expression of  $S_{c}''(\phi_{c}) $ in \eqref{2.13}, we have
\begin{align*}
	S_{c}''(\phi_{c})f=0 \Leftrightarrow L_{\omega}f =0, 
\end{align*}
for any $f \in \ker (S_{c}''(\phi_{c})), $
that is 
\begin{align}
	\label{3.2}
-\partial_{xx}f+(1-\omega^{2})f+(p+1)\psi_{\omega}^{p}f=0. 
\end{align}
By the work of Weinstein \cite{weinstein-1985-modulational}, 
the only solution to \eqref{3.2} are
\begin{align*}
	f=\alpha \partial_x \psi_{\omega}, \quad \alpha \in \R. 
\end{align*}
Note that
\begin{align*}
	\partial_x \psi_{\omega}
	=c^{-\frac{1}{p}} \partial_x \phi_{c}. 
\end{align*}
This implies that
$ f \in \{ \alpha \partial_x \phi_{c}:\alpha \in \R  \}, $
and we have
\begin{align*}
	\ker (S_{c}''(\phi_{c})) \subset 
	\{ \alpha \partial_x \phi_{c}:\alpha \in \R \}.
\end{align*}  
Finally, combining the two relationship gives us
\begin{align*}
	\ker (S_{c}''(\phi_{c}))
	= \{ \alpha \partial_x \phi_{c}:\alpha \in \R \}.
\end{align*}
This gives the proof of the lemma.
\end{proof}

The second lemma is the uniqueness of the negative eigenvalue of $S_{c}''(\phi_{c}).$
\begin{lem}
	\label{lem3.2}
	$S_{c}''(\phi_{c})$
exists only one negative eigenvalue.
\end{lem}
\begin{proof}
	It is known that the operator 
$-\partial_{xx} +(1-\omega^{2})+(p+1)\psi_{\omega}^{p}$ 
has only one negative eigenvalue(see \cite{weinstein-1985-modulational}), and we denote it by $\lambda_{-1}. $
Then there exists a unique associated eigenfunction
$ \theta \in \H1 $
such that
\begin{align}
	\label{3.3}
	-\partial_{xx} \theta +(1-\omega^{2})\theta-(p+1)\psi_{\omega}^{p}\theta
	=\lambda_{-1} \theta.
\end{align}
Using the expression of
$ S_{c}''(\phi_{c})$
in \eqref{2.13}, we have
\begin{align*}
\langle S_{c}''(\phi_{c})\partial_{\omega}\psi_{\omega}, \partial_{\omega}\psi_{\omega}\rangle 
&=-c \langle L_{\omega}\partial_{\omega}\psi_{\omega},   
   \partial_{\omega}\psi_{\omega}\rangle \\
&=-c \int_{\R} (-\partial_{xx}
    \partial_{\omega}\psi_{\omega} +(1-\omega^{2})\partial_{\omega}\psi_{\omega}-(p+1)\psi_{\omega}^{p}\partial_{\omega}\psi_{\omega})
    \cdot
    \partial_{\omega}\psi_{\omega} \dx\\
&=-c \int_{\R} 2\omega\psi_{\omega}
   \cdot
   \partial_{\omega}\psi_{\omega} \dx \\
&=-c \omega \frac{d}{d\omega}
  \Big(\big\|\psi_{\omega} \big\|^{2}_{L^{2}}\Big)\\
&=2(\frac{2}{p}-\frac{1}{2})
  (1-\omega^{2})^{\frac{2}{p}-\frac{3}{2}}
  \big\|\psi_{0} \big\|^{2}_{L^{2}}. 
\end{align*}
Note that $p>4, $
then we have
\begin{align*}
\langle S_{c}''(\phi_{c})\partial_{\omega}\psi_{\omega},   \partial_{\omega}\psi_{\omega}\rangle 
 < 0. 
\end{align*}
This implies that $S_{c}''(\phi_{c}) $
has at least one negative eigenvalue $\mu_{0}. $ Assume its associated eigenfunction $\xi_{0}, $
that is, 
\begin{align*}
S_{c}''(\phi_{c})\xi_{0}
=\mu_{0}\xi_{0}. 
\end{align*}
Using the expression of  $S_{c}''(\phi_{c}) $ in \eqref{2.13} again, the last equality yields
\begin{align*}
	-c\big[-\partial_{xx} \xi_{0} +(1-\omega^{2})\xi_{0} -(p+1)\psi_{\omega}^{p} \xi_{0}\big]
	=\mu_{0} \xi_{0}.
\end{align*} 
Then we have
$ \xi_{0}=-\frac{1}{c} \theta. $
Hence, by \eqref{3.3}, 
$(\mu_{0}, \xi_{0})$
is exactly the pair satisfying
\begin{align}
	\mu_{0}=\lambda_{-1}, \quad \xi_{0}=-\frac{1}{c} \theta.
\end{align}
This implies that $S_{c}''(\phi_{c})$
has exactly one simple negative eigenvalue. This completes the proof of Lemma \eqref{3.2}. 
\end{proof}

 \subsection{Coercivity }
 In this subsection, we give a general  coercivity property on the Hessian of the action $S_{c}''(\phi_{c}). $
% We assume that 
% \begin{itemize}
% \item[(H1)] 
% \label{H1}
% There exists a unique $\Gamma_{c}\in H^1$, such that $\ker S_{c}''(\phi_{c})
%	=\{\alpha \Gamma_{c}:\alpha \in \R \}. $
% \item[(H2)]
% \label{H2}
%  $S_{c}''(\phi_{c})$ exists only one negative eigenvalue.
% \end{itemize}
% 
% Note that for any $y\in \R$, 
% $$
% S'(\phi_c(\cdot-y))=0,
% $$
% taking derivative on $y$, we have that 
% $$
% S''(\phi_c)\partial_x\phi_c=0.
% $$
% Therefore, under the assumption of (H1), we may set that $\Gamma_c=\partial_x\phi_c$.
% 
 
%  Now we prove the following coercivity property.
\begin{prop}
	\label{prop3.4}
	Let $\tau_c, \Psi$ be any functions satisfying that  
\begin{align}
	\label{Assume-Coer}
	& \tau_{c}=S_{c}''(\phi_{c})\Psi,\quad \mbox{and}\quad \langle S_{c}''(\phi_{c})\Psi, \Psi\rangle  <0.
\end{align} 
	Suppose that $\xi\in\H1$ satisfies
	\begin{align}
		\label{3.10}
		\langle \xi, \partial_x\phi_c\rangle 
		=\langle \xi, \tau_{c}\rangle 
		=0.
	\end{align}
Then
\begin{align*}
	\langle S_{c}''(\phi_{c})\xi, \xi\rangle  
	\gtrsim \|\xi\|^{2}_{\H1}. 
\end{align*}
\end{prop}
\begin{proof}
	From the expression of $S_{c}''(\phi_{c})$ in \eqref{2.13}, we can write $S_{c}''(\phi_{c})$ as
\begin{align*}
	S_{c}''(\phi_{c})=-c(L+V), 
\end{align*}
where 
$L=-\partial_{xx}+(1-\omega^{2}), $
and
$V=-(p+1)\psi_{\omega}^{p}.$
Hence $V$ is a compact perturbation of the self-adjoint operator $L. $\\
\textit{Step 1. Analyze the spectrum of 
	$S_{c}''(\phi_{c}). $ }
  We first compute the essential spectrum of $L. $
  Note that for any $g\in \H1, $
  \begin{align}
  	\label{3.11}
 \langle Lg, g\rangle 
 &=\int_{\R} \big(-\partial_{xx}g+(1-\omega^{2})g\big) \cdot g \dx \notag \\
 &=\big\|\partial_x g\big\|^{2}_{L^2}     +(1-\omega^{2})\big\|g\big\|^2_{L^{2}}. 
  \end{align}
Since $c=\omega^{-2}, c>1, $
we can get $|\omega|<1, $ and thus
\begin{align*}
	 \langle Lg, g \rangle  
	 \gtrsim \big\|g\big\|^{2}_{\H1}.
\end{align*}
  This means that there exists $ \delta>0 $
such that the essential spectrum of $L$ is 
$[\delta, +\infty). $
By Weyl Theorem, 
$ S_{c}''(\phi_{c})$ and $L$
share the same essential spectrum.
So we obtain the essential spectrum of $S_{c}''(\phi_{c}). $
Recall that we have obtained the only one negative eigenvalue $\mu_{0}$ of $S_{c}''(\phi_{c})$ in Lemma \ref{lem3.2} and the kernel of $S_{c}''(\phi_{c})$ in Lemma \ref{lem3.1}. 
So the discrete spectrum of $S_{c}''(\phi_{c})$ is $\mu_{0}, 0, $ and the essential spectrum is 
$[\delta, +\infty). $
\\
\textit{Step 2. Positivity.}
%  The argument here is inspired by  \cite{J.bellazzini-2014-multi, S. Le Coz-2018-stability}. 
By Lemma \ref{lem3.2}, we have the unique negative eigenvalue $\mu_{0} $ and the eigenfunction $\xi_{0}$ of $S_{c}''(\phi_{c}). $
For convenience, we normalize the eigenfunction $\xi_{0}$ such that $\big\|\xi_{0}\big\|_{L^{2}}=1. $
Hence, for  $ \xi_{0}\in\H1, $
by the spectral decomposition theorem we can write the decomposition of $\xi$ along the spectrum of $S_{c}''(\phi_{c}), $
\begin{align*}
	\xi=a_{\xi}\xi_{0}
	      +b_{\xi}\partial_x\phi_{c}
	      +g_{\xi}, 
\end{align*}
where $a_{\xi}, b_{\xi}\in \R, $
and $g_{\xi}$ lies in the positive eigenspace of $S_{c}''(\phi_{c}), $
that is, $g_{\xi}$ satisfies
\begin{align*}
	\langle g_{\xi}, \xi_{0}\rangle 
	=\langle g_{\xi}, \partial_x\phi_{c}\rangle 
	=0, 
\end{align*}
and there exists an absolute constant $\sigma>0$ such that
\begin{align}
	\label{3.12}
	\langle S_{c}''(\phi_{c})g_{\xi}, g_{\xi}\rangle 
	\geqslant
	\sigma \big\|g_{\xi}\big\|^{2}_{L^{2}}.
\end{align}
Since $\xi$ satisfies the orthogonality condition
$\langle\xi, \partial_x\phi_{c}\rangle $
in \eqref{3.10} and
$\langle\xi_{0}, \partial_x\phi_{c} \rangle =0, $
we have
$ b_{\xi}=0, $
and thus
\begin{align}
	\label{3.13}
	\xi=a_{\xi}\xi_{0}+g_{\xi}. 
\end{align}
Substituting \eqref{3.13} into 
$\langle S_{c}''(\phi_{c})\xi, \xi\rangle , $
we get
\begin{align*}
\langle S_{c}''(\phi_{c})\xi, \xi\rangle 
&=\langle S_{c}''(\phi_{c})(a_{\xi}\xi_{0}+g_{\xi}),      a_{\xi}\xi_{0}+g_{\xi}\rangle \notag\\
&=a_{\xi}^{2}\langle S_{c}''(\phi_{c})\xi_{0},  \xi_{0}\rangle 
  +2\mu_{0}a_{\xi}\langle \xi_{0},g_{\xi}\rangle 
   +\langle S_{c}''(\phi_{c})g_{\xi}, g_{\xi}\rangle.   
\end{align*}
Due to the orthogonality property of
$\langle \xi_{0}, g_{\xi}\rangle =0, $ 
we have
\begin{align}
	\label{3.14}
	\langle S_{c}''(\phi_{c})\xi, \xi\rangle 
=\mu_{0}a_{\xi}^{2}
 +\langle S_{c}''(\phi_{c})g_{\xi}, g_{\xi}\rangle.
\end{align}
To $\Psi, $ by spectral decomposition theorem again, we may write
\begin{align*}
\Psi=a\xi_{0}+b\partial_x\phi_{c}+g, 
\end{align*}
where $a, b\in\R, $
and $g$ lies in the positive eigenspace of $S_{c}''(\phi_{c}). $
We note that
\begin{align}
	\label{S_c-psi}
	S''_c(\phi_{c})\Psi=S''_c(\phi_{c})(a\xi_{0}+b\partial_x\phi_{c}+g)=S''_c(\phi_{c})(a\xi_{0}+g).
\end{align} 

Therefore, a similar computation as above shows that
\begin{align*}
\langle S_{c}''(\phi_{c})\Psi, \Psi\rangle 
&=\langle S_{c}''(\phi_{c})(a\xi_{0}+g), (a\xi_{0}+g)\rangle  \\
&=\mu_{0}a^{2}+\langle S_{c}''(\phi_{c})g, g\rangle.
\end{align*}
For convenience, let $ -\delta_{0}=\langle S_{c}''(\phi_{c})\Psi, \Psi\rangle. $
Then by \eqref{Assume-Coer}, we know that $\delta_{0}>0. $
Moreover, we have
\begin{align}
	\label{3.15}
	-\delta_{0}
	=\mu_{0}a^{2}+\langle S_{c}''(\phi_{c})g, g\rangle.
\end{align}
By \eqref{3.13} and \eqref{S_c-psi}, using the orthogonality assumption
$\langle \xi, \tau_c \rangle =0 $ in \eqref{3.10} 
we have
\begin{align*}
0=\langle \xi, \tau_c\rangle 
&=\langle a_{\xi}\xi_{0}+g_{\xi}, S_{c}''(\phi_{c})\Psi\rangle  \\
&=\langle a_{\xi}\xi_{0}+g_{\xi}, S_{c}''(\phi_{c})(a\xi_{0}+g)\rangle   \\
&=\langle a_{\xi}\xi_{0}, S_{c}''(\phi_{c})a\xi_{0}\rangle 
  +\langle S_{c}''(\phi_{c})g, g_{\xi}\rangle  \\
&=\mu_{0}aa_{\xi}  \langle \xi_{0},\xi_{0}\rangle 
  +\langle S_{c}''(\phi_{c})g, g_{\xi}\rangle \\
&=\mu_{0}aa_{\xi}
  +\langle S_{c}''(\phi_{c})g, g_{\xi}\rangle. 
\end{align*}
So we get the equality
\begin{align*}
	0=\mu_{0}aa_{\xi}
	    +\langle S_{c}''(\phi_{c})g, g_{\xi}\rangle. 
\end{align*}
By the Cauchy-Schwartz inequality, we have
\begin{align*}
	(\mu_{0}aa_{\xi})^{2}
	&=\langle S_{c}''(\phi_{c})g, g_{\xi}\rangle ^{2}\\
	&\leqslant \langle S_{c}''(\phi_{c})g, g\rangle 
	 \langle S_{c}''(\phi_{c})g_{\xi}, g_{\xi}\rangle.
\end{align*}
This gives
\begin{align}
	\label{3.16}
	(-\mu_{0}a^{2})(-\mu_{0}a_{\xi}^{2})
    \leqslant \langle S_{c}''(\phi_{c})g, g\rangle 
	\langle S_{c}''(\phi_{c})g_{\xi}, g_{\xi}\rangle.
\end{align}
The last inequality combining with \eqref{3.15} implies that
\begin{align*}
	-\mu_{0}a_{\xi}^{2}
&\leqslant 
   \frac{\langle S_{c}''(\phi_{c})g, g\rangle 
	\langle S_{c}''(\phi_{c})g_{\xi}, g_{\xi}\rangle }
     {-\mu_{0}a^{2}}\\
&=\frac{\langle S_{c}''(\phi_{c})g, g\rangle 
	\langle S_{c}''(\phi_{c})g_{\xi}, g_{\xi}\rangle }
    {\delta_{0}+\langle S_{c}''(\phi_{c})g, g\rangle },
\end{align*}
that is 
\begin{align}
	\label{3.17}
	\mu_{0}a_{\xi}^{2}
	\geqslant
	-\frac{\langle S_{c}''(\phi_{c})g, g\rangle 
		\langle S_{c}''(\phi_{c})g_{\xi}, g_{\xi}\rangle }
	{\delta_{0}+\langle S_{c}''(\phi_{c})g, g\rangle }. 
\end{align}
Inserting \eqref{3.17} into \eqref{3.14}, we obtain
\begin{align*}
	\langle S_{c}''(\phi_{c})\xi, \xi\rangle 
&\geqslant
  -\frac{\langle S_{c}''(\phi_{c})g, g\rangle 
	\langle S_{c}''(\phi_{c})g_{\xi}, g_{\xi}\rangle }
    {\delta_{0}+\langle S_{c}''(\phi_{c})g, g\rangle }
  +\langle S_{c}''(\phi_{c})g_{\xi}, g_{\xi}\rangle \\
&=\langle S_{c}''(\phi_{c})g_{\xi}, g_{\xi}\rangle 
  \left(1-
  \frac{\langle S_{c}''(\phi_{c})g, g\rangle }
       {\delta_{0}+\langle S_{c}''(\phi_{c})g, g\rangle }\right)\\
&=\langle S_{c}''(\phi_{c})g_{\xi}, g_{\xi}\rangle 
  \frac{\delta_{0}}{\delta_{0}+\langle S_{c}''(\phi_{c})g, g\rangle }. 
\end{align*}
Recalling that $g_{\xi}$ satisfies \eqref{3.12}, we have
\begin{align}
	\label{3.18}
	\langle S_{c}''(\phi_{c})\xi, \xi\rangle 
\geqslant
  \frac{\delta_{0} \sigma}
       {\delta_{0}+\langle S_{c}''(\phi_{c})g, g\rangle }
  \big\|g_{\xi}\big\|^{2}_{L^{2}}, 
  \quad \sigma>0. 
\end{align}
From the expression of $\xi$ in \eqref{3.13} and the inequality \eqref{3.14}, 
we have
\begin{align*}
	\|\xi\|^{2}_{L^{2}}
 &=\big\|a_{\xi}\xi_{0}+g_{\xi}\big\|^{2}_{L^{2}}
  =a_{\xi}^{2}+\big\|g_{\xi}\big\|^{2}_{L^{2}}\\
 &\leqslant
 -\frac{\langle S_{c}''(\phi_{c})g, g\rangle 
 	    \langle S_{c}''(\phi_{c})\xi, \xi\rangle }
       {\mu_{0}\delta_{0}}
 +\big\|g_{\xi}\big\|^{2}_{L^{2}}\\
&\lesssim
  \langle S_{c}''(\phi_{c})\xi, \xi\rangle    
\end{align*}
Therefore, this gives
\begin{align}
	\label{3.19}
	\langle S_{c}''(\phi_{c})\xi, \xi\rangle 
\gtrsim
    \|\xi\|^{2}_{L^{2}}. 
\end{align}
  To obtain the final conclusion, we still need to estimate
\begin{align*}
	\langle S_{c}''(\phi_{c})\xi, \xi\rangle 
	\gtrsim
	\|\xi\|^{2}_{\H1}. 
\end{align*}
Using the expression of $S_{c}''(\phi_{c})$ in \eqref{2.8}, we have
\begin{align*}
	\langle S_{c}''(\phi_{c})\xi, \xi\rangle 
&=\int_{\R}
    (c\partial_{xx} \xi +(1-c)\xi+(p+1)\phi_{c}^{p}\xi)
     \cdot
     \xi \dx \\
&=-c\big\|\partial_x\xi\big\|^{2}_{L^{2}}+(1-c)\|\xi\|^{2}_{L^{2}}   +(p+1)\int_{\R} |\phi_{c}|^{p}\xi^{2} \dx  
\end{align*}
Thus by \eqref{3.19}, we get
\begin{align}
	\label{3.20}
	\big\|\partial_x\xi\big\|^{2}_{L^{2}}
&=-\frac{1}{c}
 \left[\langle S_{c}''(\phi_{c})\xi, \xi\rangle 
   -(1-c)\|\xi\|^{2}_{L^{2}}
   -(p+1)\int_{\R}|\phi_{c}|^{p}\xi^{2} \dx \right]\notag\\
&\leqslant
 -\frac{1}{c}\langle S_{c}''(\phi_{c})\xi, \xi\rangle 
 +(\frac{1}{c}-1)\|\xi\|^{2}_{L^{2}}
 +\frac{p+1}{c}\big\|\phi_{c}\big\|^{p}_{L^{\infty}}
   \|\xi\|^{2}_{L^{2}} \notag\\
&\leqslant
 -\frac{1}{c}\langle S_{c}''(\phi_{c})\xi, \xi\rangle 
 +\left(\frac{1}{c}-1+\frac{p+1}{c}\big\|\phi_{c}\big\|^{p}_{L^{\infty}}\right) 
  \|\xi\|^{2}_{L^{2}} \notag\\
&\lesssim 
 \langle S_{c}''(\phi_{c})\xi, \xi\rangle 
  + \|\xi\|^{2}_{L^{2}} \notag\\
&\lesssim 
 \langle S_{c}''(\phi_{c})\xi, \xi\rangle. 
\end{align}
Therefore, together \eqref{3.19} and \eqref{3.20}, we obtain
\begin{align*}
	 \|\xi\|^{2}_{\H1}
  &= \|\xi\|^{2}_{L^{2}}+\big\|\partial_x\xi\big\|^{2}_{L^{2}}\\
  &\lesssim \langle S_{c}''(\phi_{c})\xi, \xi\rangle.
\end{align*}
Thus we obtain the desired result.
\end{proof}

\begin{cor}
	\label{cor:a.4}
	Assume
	\begin{align}
		\label{condition_1}
		\langle  S''_c(\phi_{c})\eta,  \eta \rangle \leqslant 0,  \quad  \eta \notin \ker S''_c(\phi_{c}), \quad \eta\in\H1, 
	\end{align}
then for any $\zeta \in \H1, $ s.t.
\begin{align}
	\label{orth-condition 2}
	\langle S_c''(\phi_{c})\zeta, \eta\rangle =0,
\end{align}
we have
\begin{align*}
	\langle S_c''(\phi_{c})\zeta, \zeta\rangle \geqslant 0.
\end{align*}
\end{cor}
\begin{proof}
	Using the similar spectral decomposition argument as in Proposition \ref{prop3.4}, we apply the notation from Proposition \ref{prop3.4}, that is: the unique negative eigenvalue $\mu_{0}$ and its corresponding normalized eigenfunction $\xi_{0}$ of $S_c''(\phi_{c}). $ So for $\eta \in \H1, $ we can write the decomposition of $\eta$ as
	\begin{align}
		\label{decomposition_eta}
		\eta=a_{\eta}\xi_{0}+b_{\eta}\partial_{x}\phi_{c}+g_{\eta},
	\end{align}
where $a_{\eta}, b_{\eta} \in \R, $ and $g_{\eta}$ lies in the positive eigenspace of $S''_c(\phi_{c}), $ that is $g_{\eta}$
satisfies 
\begin{align}
	\label{orth-property}
	\langle g_{\eta}, \xi_{0}\rangle=\langle g_{\eta}, \partial_{x}\phi_{c}\rangle=0,
\end{align}
and there exists an absolute constant $\sigma_1>0$ such that 
\begin{align}
	\label{positive-eigen1}
	\langle S''_c(\phi_{c})g_{\eta}, g_{\eta} \rangle\geqslant \sigma_1\|g_{\eta}\|^{2}_{L^{2}}.
\end{align}
Since $\eta$ satisfies $\eqref{condition_1}, $ there exists an absolute constant $\delta_1\geqslant0, $ such that 
\begin{align}
	\langle S''_c(\phi_{c})\eta, \eta \rangle=-\delta_1.
\end{align}
By Lemma \ref{lem3.1}, and combining \eqref{decomposition_eta} and \eqref{orth-property}, we have
\begin{align*}
S''_c(\phi_{c})\eta
=a_{\eta}S''_c(\phi_{c})\xi_{0}+S''_c(\phi_{c})g_{\eta}.
\end{align*}
So we obtain that 
\begin{align}
	\label{negative property of eta}
	\langle S''_c(\phi_{c})\eta, \eta \rangle
	&=\langle a_{\eta}S''_c(\phi_{c})\xi_{0}+S''_c(\phi_{c})g_{\eta}, a_{\eta}\xi_{0}+b_{\eta}\partial_{x}\phi_{c}+g_{\eta}\rangle
	\notag\\ 
	&=\mu_{0}a_{\eta}^2+\langle S''_c(\phi_{c})g_{\eta}, g_{\eta} \rangle
	=-\delta_1.
\end{align}
Similarly, we write $\zeta$ as $\zeta=a_1\xi_{0}+b_1\partial_{x}\phi_{c}+g_1, $
where $a_1, b_1\in \R, $ and $g_1$ lies in the positive eigenspace of $ S''_c(\phi_{c}). $
We note that 
\begin{align}
	\label{S_c(phi_c)zeta}
\langle S''_c(\phi_{c})\zeta,\zeta \rangle=\mu_{0}a_{1}^2+\langle S''_c(\phi_{c})g_{1}, g_1\rangle.
\end{align}
From condition \eqref{orth-condition 2}, we have
\begin{align*}
	\langle S''_c(\phi_{c})\zeta, \eta \rangle
	&=\langle a_{1}S''_c(\phi_{c})\xi_{0}+S''_c(\phi_{c})g_{1}, a_{\eta}\xi_{0}+b_{\eta}\partial_{x}\phi_{c}+g_{
	\eta}\rangle
	\\
	&=\mu_{0}a_{1}a_{\eta}+\langle S''_c(\phi_{c})g_{1}, g_{\eta}\rangle=0.
\end{align*}
So we get 
\begin{align*}
	\big(\mu_{0}a_{1}a_{\eta}\big)^2
	&=\big(\mu_{0}a_{1}^2\big)\big(\mu_{0}a_{\eta}^2\big)
	\\
	&=\langle S''_c(\phi_{c})g_{1}, g_{\eta}\rangle^2
	\\
	&\leqslant
	\langle S''_c(\phi_{c})g_{1}, g_{1}\rangle \langle S''_c(\phi_{c})g_{\eta}, g_{\eta}\rangle,
\end{align*}
where we used Cauchy-Schwartz inequality in the last step.
Combining \eqref{negative property of eta}, the last inequality implies that 
\begin{align}
	\label{lemma a.4-1}
	\mu_{0}a_{1}^2
	&\geqslant
	\frac{\langle S''_c(\phi_{c})g_{1}, g_{1}\rangle \langle S''_c(\phi_{c})g_{\eta}, g_{\eta}\rangle}{\mu_{0}a_{\eta}^2}
	\notag\\
	&=
	-\frac{\langle S''_c(\phi_{c})g_{1}, g_{1}\rangle \langle S''_c(\phi_{c})g_{\eta}, g_{\eta}\rangle}
	{\langle S''_c(\phi_{c})g_{\eta}, g_{\eta} \rangle+\delta_1}.
\end{align} 
Inserting \eqref{lemma a.4-1} into \eqref{S_c(phi_c)zeta}, we have
\begin{align*}
	\langle S''_c(\phi_{c})\zeta,\zeta \rangle
	&\geqslant
		-\frac{\langle S''_c(\phi_{c})g_{1}, g_{1}\rangle \langle S''_c(\phi_{c})g_{\eta}, g_{\eta}\rangle}
	{\langle S''_c(\phi_{c})g_{\eta}, g_{\eta} \rangle+\delta_1}
	+\langle S''_c(\phi_{c})g_{1}, g_{1}\rangle
	\\
&=\frac{\delta_1\langle S''_c(\phi_{c})g_{1}, g_{1}\rangle}{\langle S''_c(\phi_{c})g_{\eta}, g_{\eta} \rangle+\delta_1}
\geqslant0.
\end{align*}
Thus we complete the proof.

\end{proof}

\subsection{Modulation}
The modulation theory shows that by choosing suitable parameters, some orthogonality conditions as in \eqref{prop3.4} can be verified. 
%The modulation is obtained via the standard implicit function theorem. We assume that
%\begin{itemize}	
%	\item[(H3)] Let $\tau_c$ be the function satisfying
%	\begin{align}
%			\left| 
%		\begin{array}{cccc}
%			0 & \big\| \partial_x \phi_{c}\big\|_{L^2}^2 \\
%			-\langle \partial_{c}\phi_{c}, \tau_c\rangle & \langle \partial_{x}\phi_{c}, \tau_c\rangle
%		\end{array}
%		\right|
%		\neq 0.
%	\end{align}
%\end{itemize}
%\begin{remark}
%	If $\tau_c $ is an even equation, then (H3) can be simplified as
%	$\langle \partial_{c}\phi_{c}, \tau_c\rangle \neq 0.$
%\end{remark}

\begin{prop}
	\label{prop4.1}
	Assume that $\tau_c$ be the function satisfying
	\begin{align}\label{Assume-Modulation}
	\langle \partial_{c}\phi_{c}, \tau_c\rangle \neq 0.
	\end{align}
	Moreover, suppose that there exists $\varepsilon_{0}>0$ such that for any 
	$\varepsilon \in (0, \varepsilon_{0})$, and any $ 
	  u \in U_{\varepsilon}(\phi_{c}), $
	then the following properties are verified. There exist $C^{1}$-functions 
\begin{align*}
y: U_{\varepsilon}(\phi_{c}) \rightarrow \R, \quad \lambda: U_{\varepsilon}(\phi_{c}) \rightarrow \R^{+}
\end{align*}
such that if we define $\xi$ by 
\begin{align}
	\label{4.2}
	\xi=u(\cdot +y)-\phi_{\lambda}, 
\end{align}
then  $\xi$ satisfies the following orthogonality conditions: 
\begin{align}
	\label{4.3}
	\langle \xi, \partial_x\phi_{\lambda}\rangle 
  =	\langle \xi, \tau_{\lambda}\rangle 
  =0. 
\end{align}
\end{prop}
\begin{proof}
	We use the Implicit Function Theorem to prove this proposition. Here we only give the important steps of the proof and refer the readers to \cite{weinstein-1985-modulational,weinstein-1986-lyapunov} for the similar argument. Define
\begin{align*}
	\vec p=(u; \lambda, y), \qquad \vec p_{0}=(\phi_{c}; c, 0).
\end{align*}
Let $\varepsilon$ be the parameter decided later, and define the functional pair 
$(F_{1}, F_{2}): 
  U_{\varepsilon}(\phi_{c})\times \R^{+} \times \R \rightarrow \R^{2}$
 as
\begin{align*}
	F_{1}(\vec p)=\langle  \xi, \partial_x\phi_{\lambda}\rangle , \quad
	F_{2}(\vec p)=\langle  \xi, \tau_{\lambda}\rangle.
\end{align*}
We claim that there exists $\varepsilon_{0}>0, $ such that for any $\varepsilon \in (0, \varepsilon_{0}), $ there exists a unique $C^{1}$ map:
$ U_{\varepsilon}(\phi_{c}) \rightarrow  \R^{+} \times \R $
such that 
$ (F_{1}(\vec p), F_{2}(\vec p))=0. $
  Indeed, we have
 \begin{align*}
 	F_{1}(\vec p_{0})= F_{2}(\vec p_{0})=0.
 \end{align*}
Second, we prove that
\begin{align*}
	|J|=
	\left| 
	\begin{array}{cccc}
		\partial_{\lambda}F_{1} & \partial_{y}F_{1}\\
		\partial_{\lambda}F_{2} & \partial_{y}F_{2}
	\end{array}
     \right|_{\vec p=\vec p_{0}}
     \neq 0.
\end{align*}
Indeed, a direct computation gives that
\begin{align*}
	\partial_{\lambda}F_{1}(\vec p)
	=\partial_{\lambda}\langle \xi, \partial_x\phi_{\lambda}\rangle 
   &=\partial_{\lambda}\langle u(t, x+y(t))-\phi_{\lambda}, \partial_x\phi_{\lambda}\rangle  \\
   &=\langle -\partial_{\lambda}\phi_{\lambda}, \partial_x\phi_{\lambda}\rangle 
    +\langle u(t,x+y(t))-\phi_{\lambda(t)}, \partial_{\lambda}\partial_x\phi_{\lambda}\rangle.
\end{align*}
When $\vec p=\vec p_{0}, $ we observe that
$u(t,x+y(t))-\phi_{\lambda(t)}=0, $ and the second term vanishes. So we get
\begin{align*}
		\partial_{\lambda}F_{1}(\vec p)\big|_{\vec p=\vec p_{0}} =-\langle \partial_{c}\phi_{c}, \partial_x\phi_{c}\rangle=0 
\end{align*}
as $\phi_{c}$ is an even function.
A similar computation shows that
\begin{align*}
	\partial_{y}F_{1}(\vec p)\big|_{\vec p=\vec p_{0}}
  &=\langle \partial_x u(t,x+y(t)),\partial_x\phi_{\lambda} \rangle \big|_{\vec p=\vec p_{0}}
  =\big\| \partial_x \phi_{c}\big\|_{L^2}^2;
\\ 
    \partial_{\lambda}F_{2}(\vec p)\big|_{\vec p=\vec p_{0}}
  &=\partial_{\lambda}\langle u(t,x+y(t))-\phi_{\lambda},\tau_{\lambda} \rangle \big|_{\vec p=\vec p_{0}}
  =-\langle \partial_{c}\phi_{c}, \tau_c \rangle;
\\
  	\partial_{y}F_{2}(\vec p)\big|_{\vec p=\vec p_{0}}
  &=\langle \partial_x u(t,x+y(t)), \tau_{\lambda} \rangle \big|_{\vec p=\vec p_{0}}
  =\langle \partial_{x}\phi_{c}, \tau_c\rangle. 
\end{align*}
By \eqref{Assume-Modulation}, we find that
\begin{align*}
	\left| 
	\begin{array}{cccc}
		\partial_{\lambda}F_{1} & \partial_{y}F_{1}\\
		\partial_{\lambda}F_{2} & \partial_{y}F_{2}
	\end{array}
	\right|_{\vec p=\vec p_{0}}
  &=	\left| 
  	\begin{array}{cccc}
     0 & \big\| \partial_x \phi_{c}\big\|_{L^2}^2 \\
  	-\langle \partial_{c}\phi_{c}, \tau_c \rangle & \langle \partial_{x}\phi_{c}, \tau_c\rangle
  	\end{array}
  \right|
   \\
  &=\big\| \partial_x \phi_{c}\big\|_{L^2}^2 \langle \partial_{c}\phi_{c}, \tau_c\rangle
  \neq 0.
\end{align*}
Therefore, the Implicit Function Theorem implies that there exists $\varepsilon_{0}>0$ such that for 
$\varepsilon \in (0, \varepsilon_{0}), 
  u\in U_{\varepsilon}(\phi_{c}), $ there exist unique $C^{1}$-functions
\begin{align*}
	y: U_{\varepsilon}(\phi_{c}) \rightarrow \R, \quad
	\lambda:  U_{\varepsilon}(\phi_{c}) \rightarrow \R^{+}, 
\end{align*}
such that
\begin{align}
	\label{4.5}
	\langle \xi, \partial_x\phi_{\lambda}\rangle 
  =\langle \xi, \tau_{\lambda}\rangle 
  =0.
\end{align}
This proves the Proposition.
\end{proof}

\subsection{The negativity of $\langle S''_c(\phi_{c})\Gamma_c, \Gamma_c\rangle$ (numerically checked).}
\label{numerical result}
We recall the expression of $\Gamma_c$ and $\kappa_c$ which introduced in \eqref{Gamma_c} and \eqref{kappa_c}: 
\begin{align*}
	&\Gamma_c
	=B(c)\big(c^2\Psi_{c}+\frac{c}{2}x\partial_{x}\phi_{c}+c\phi_{c}\big)+D(c)(3x^2\phi_{c}+x^3\partial_{x}\phi_c),
	\\
     &S''_c(\phi_{c})\Gamma_c=\kappa_c.
\end{align*}

$\bullet$\textit{The expression of $\kappa_c.$ } 
From Lemma \ref{lem3.3}, we have already known that 
\begin{align}
	\label{f_1}
	S''_c(\phi_{c})\Psi_{c}=\phi_{c},\\
	\label{f_2}
	S''_c(\phi_{c})(\frac{1}{2c}x\partial_{x}\phi_{c})=\partial_{xx}\phi_{c}.
\end{align}
By the expression of $	S''_c(\phi_{c})$ in \eqref{2.8}, we have
\begin{align*}
	S''_c(\phi_{c})\phi_{c}=c\partial_{xx}\phi_{c}+(1-c)\phi_{c}+(p+1)\phi_{c}^{p+1}.
\end{align*}
From equation \eqref{1.3}, we have
\begin{align}
	\label{phi_c^p+1}
	\phi_{c}^{p+1}=-c\partial_{xx}\phi_{c}+(c-1)\phi_{c}.
\end{align}
Thus we obtain
\begin{align}
	\label{S''phi_c}
		S''_c(\phi_{c})\phi_{c}=-pc\partial_{xx}\phi_{c}+p(c-1)\phi_{c}.
\end{align}
Using the expression of $	S''_c(\phi_{c})$ in \eqref{2.8} again, we obtain
\begin{align}
	\label{S''f_4-step1}
	&\quad
S''_c(\phi_{c})(3x^2\phi_{c}+x^3\partial_{x}\phi_c)\notag\\
&=c\partial_{xx}(3x^2\phi_{c}+x^3\partial_{x}\phi_c)+(1-c)(3x^2\phi_{c}+x^3\partial_{x}\phi_c)+(p+1)\phi_{c}^p(3x^2\phi_{c}+x^3\partial_{x}\phi_c)\notag\\
&=c(6\phi_{c}+18x\partial_{x}\phi_{c}+9x^2\partial_{xx}\phi_{c}+x^3\partial^3_{x}\phi_{c})+(1-c)(3x^2\phi_{c}+x^3\partial_{x}\phi_c)+(p+1)\phi_{c}^p(3x^2\phi_{c}+x^3\partial_{x}\phi_c)\notag\\
&=6c\phi_{c}+18cx\partial_{x}\phi_{c}+6cx^2\partial_{xx}\phi_{c}+3x^2\big[c\partial_{xx}\phi_{c}+(1-c)\phi_{c}+(p+1)\phi_{c}^{p+1}\big]
\notag\\
&\quad
+x^3\partial_{x}(c\partial_{xx}\phi_{c}+(1-c)\phi_{c}+\phi_{c}^{p+1}).
\end{align}
Inserting \eqref{phi_c^p+1} into \eqref{S''f_4-step1}, and by \eqref{1.3}, we have
\begin{align}
	\label{S''f_4}
	S''_c(\phi_{c})(3x^2\phi_{c}+x^3\partial_{x}\phi_c)
	=6c\phi_{c}+18cx\partial_{x}\phi_{c}+(6c-3pc)x^2\partial_{xx}\phi_{c}+3p(c-1)x^2\phi_{c}.
\end{align}
Combining \eqref{f_1}, \eqref{f_2}, \eqref{S''phi_c} and \eqref{S''f_4}, we finally obtain the concrete expression of $\kappa_c, $ that is 
\begin{align}
	\label{expression-kappa_c}
	\kappa_c&=S''_c(\phi_{c})\Gamma_c\notag\\
	&=\big[B(c)(p+1)c^2-B(c)pc+6cD(c)\big]\phi_{c}+B(c)(1-p)c^2\partial_{xx}\phi_{c}+18cD(c)x\partial_{x}\phi_{c}
	\notag\\
	&\quad
	+(6c-3pc)D(c)x^2\partial_{xx}\phi_{c}+3p(c-1)D(c)x^2\phi_{c}.
\end{align}

$\bullet$\textit{The numerical result of $\langle \kappa_{c}, \Gamma_{c} \rangle.$ }
According to \cite{Pego-1991-eigenvalue}, the solution of elliptic equation \eqref{1.3} $\phi_{c}$ is explicitly given by
\begin{align}
	\label{phi_c}
	\phi_{c}(x)=\left[\frac{1}{2}(c-1)(p+2)\right]^{\frac{1}{p}}\sech^{\frac{2}{p}}\big(\frac{1}{2}xp\sqrt{\frac{c-1}{c}}\big).
\end{align}
By \eqref{1.4} and the expression of $\Psi_{c}$ in \eqref{Psi_c}, we have
\begin{align*}
	\Psi_{c}
	&=-\frac{1}{2}\omega^{1-\frac{2}{p}}\partial_{\omega}\psi_{\omega}
	=-\frac{1}{2}\omega^{1-\frac{2}{p}}\frac{d}{dc}(c^{-\frac{1}{p}}\phi_{c})\cdot\frac{dc}{d\omega}
	\\
	&=c^{1+\frac{1}{p}}\frac{d}{dc}(c^{-\frac{1}{p}}\phi_{c}).
\end{align*}
Using \eqref{phi_c}, and by direct computation we obtain that
\begin{align}
	\label{num-Psi_c}
		\Psi_{c}
		&=\phi_{c}\cdot \Big[\frac{1}{p(c-1)}-\frac{x}{2\sqrt{c(c-1)}}\tanh\big(\frac{1}{2}xp\sqrt{\frac{c-1}{c}}\big)\Big],
		\\
\label{partial_xphi_{c}}
\partial_x\phi_{c}
&=-\sqrt{\frac{c-1}{c}}\phi_{c}\cdot \tanh\big(\frac{1}{2}xp\sqrt{\frac{c-1}{c}}\big),
\\
\label{partial_xxphi_c}
\partial_{xx}\phi_c
&=\frac{c-1}{c}\phi_{c}\cdot
 \Big[\tanh^2\big(\frac{1}{2}xp\sqrt{\frac{c-1}{c}}\big)-\frac{1}{2}p \sech ^{2}\big(\frac{1}{2}xp\sqrt{\frac{c-1}{c}}\big)\Big].
\end{align}

We check $\langle S''_c(\phi_{c})\Gamma_{c},\Gamma_{c}\rangle <0$  in the following case:
\\
{\it The limit of integration is $[-50\pi, 50\pi]. $}

(1) for $p=4.1, \langle S''_c(\phi_{c})\Gamma_{c}, \Gamma_{c}\rangle=-1024.83, $

(2) for $p=4.5, \langle  S''_c(\phi_{c})\Gamma_{c}, \Gamma_{c}\rangle=-362.82, $

(3) for $p=5, \langle  S''_c(\phi_{c})\Gamma_{c}, \Gamma_{c}\rangle=-292.10, $

(4) for $p=6, \langle   S''_c(\phi_{c})\Gamma_{c}, \Gamma_{c}\rangle=-274.60, $

(5) for $p=6.5, \langle S''_c(\phi_{c})\Gamma_{c}, \Gamma_{c}\rangle=-276.36, $

(6) for $p=10, \langle   S''_c(\phi_{c})\Gamma_{c}, \Gamma_{c}\rangle=-303.22, $

(7) for $p=30, \langle  S''_c(\phi_{c})\Gamma_{c}, \Gamma_{c}\rangle=-445.07, $

(8) for $p=50, \langle   S''_c(\phi_{c})\Gamma_{c}, \Gamma_{c}\rangle=-609.47, $

(9) for $p=70, \langle  S''_c(\phi_{c})\Gamma_{c}, \Gamma_{c}\rangle=-790.46, $

(10) for $p=100, \langle  S''_c(\phi_{c})\Gamma_{c}, \Gamma_{c}\rangle=-1083.61. $
\\

The graph of $\langle  S''_c(\phi_{c})\Gamma_{c}, \Gamma_{c}\rangle$ as a function of p is shown as below:
\begin{figure}[H]
	\centering \includegraphics[width=0.8\linewidth]{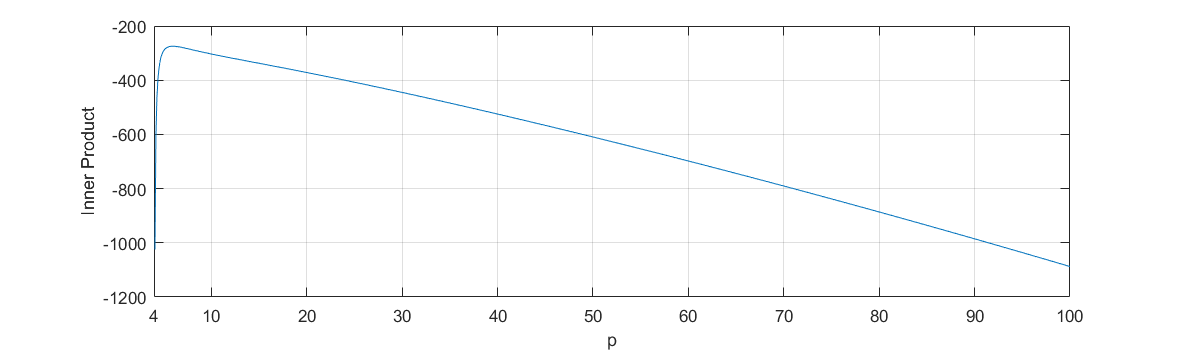}
	\caption{The negativity of $\langle S''_c(\phi_{c})\Gamma_{c}, \Gamma_{c}\rangle$. }
\end{figure}

\section*{Acknowledgment}
R. Jia and Y. Wu are partially supported by NSFC 12171356. The authors would like to thank Professor Zeng Chongchun for many valuable discussions and suggestions.

\section*{Data Availability}
There is no additional data associated to this article.

\section*{Declarations}
\subsection*{Conflict of interest}
On behalf of all authors, the corresponding author states that there is no conflict of
interest. Data sharing is not applicable to this article as no datasets were generated or analysed during the
current study.

\end{document}